\documentclass[12pt]{article}
\usepackage{e-jc}
\pdfoutput=1
\usepackage{amsmath,amssymb}
\usepackage{graphicx}
\usepackage[colorlinks=true,citecolor=black,linkcolor=black,urlcolor=blue]{hyperref}
\usepackage{pdfpages}
\usepackage{kbordermatrix}

\dateline{Apr 10, 2017}{Aug 17, 2017}{TBD}
\Copyright{  The authors. Released under the CC BY-ND license (International 4.0).}

\newenvironment{subproof}{\begin{proof}[Proof of 16.]}{\end{proof}}
\DeclareMathOperator{\cl}{cl}
\DeclareMathOperator{\fcl}{fcl}
\DeclareMathOperator{\si}{si}
\DeclareMathOperator{\co}{co}
\DeclareMathOperator{\gf}{GF}
\newcommand{\del}{\hspace{-0.5pt}\backslash}
\newcommand{\delete}{\del}

\newcommand{\dy}{\ensuremath{\Delta}-\ensuremath{\nabla}}

\title{Relaxations of $\gf(4)$-representable matroids}

\author{Ben Clark\\
\small\tt clarkbenj@myvuw.ac.nz\\
\and
James Oxley \qquad  Stefan H.M. van Zwam \thanks{Supported by the National Science Foundation, grant 1500343}\\
\small Department of Mathematics\\[-0.8ex]
\small Louisiana State University\\[-0.8ex]
\small Baton Rouge, LA, United States\\
\small\tt oxley@math.lsu.edu, stefanvanzwam@gmail.com}

\begin{document}

\maketitle

%
%


\begin{abstract}
We consider the $\gf(4)$-representable matroids with a circuit-hyperplane such that the matroid obtained by relaxing the circuit-hyperplane is also $\gf(4)$-represen-table. We characterize the structure of these matroids as an application of structure theorems for the classes of $U_{2,4}$-fragile and $\{U_{2,5},U_{3,5}\}$-fragile matroids. In addition, we characterize the forbidden submatrices in $\gf(4)$-representations of these matroids.
\end{abstract}

\maketitle

\section{Introduction}

Lucas \cite{lucas1975weak} determined the binary matroids that have a circuit-hyperplane whose relaxation yields another binary matroid. Truemper \cite{truemper1982alpha}, and independently, Oxley and Whittle \cite{oxley1998weak}, did the same for ternary matroids. In this paper, we solve the corresponding problem for quaternary matroids. We give both a structural characterization and a characterization in terms of forbidden submatrices. 

Truemper \cite{truemper1982alpha} used the structure of circuit-hyperplane relaxations of binary and ternary matroids to give new proofs of the excluded-minor characterizations for the classes of binary, ternary, and regular matroids. It is natural to ask if Truemper's techniques can be extended to give excluded-minor characterizations for classes of quaternary matroids. The main results of this paper can be viewed as a first step towards answering this question. 

Our structural characterization can be summarized as follows. A matroid has \textit{path width $3$} if there is an ordering $(e_1, e_2, \ldots, e_n)$ of its ground set such that $\{e_1, e_2, \ldots, e_t\}$ is a $3$-separating set for all $t\in \{1,2,\ldots, n\}$.

\begin{theorem}
\label{main-fragile2}
Let $M$ and $M'$ be $\gf(4)$-representable matroids such that $M'$ is obtained from $M$ by relaxing a circuit-hyperplane. Then $M'$ has path width $3$.
\end{theorem}

In fact, our main result, Theorem \ref{main}, describes precisely how the matroids in Theorem \ref{main-fragile2} of path width $3$ can be constructed using the \textit{generalized $\Delta$-$Y$ exchange} of \cite{oxley2000generalized} and the notion of \textit{gluing a wheel onto a triangle} from \cite{chun2013fan}. Our description uses the structure of $U_{2,4}$-fragile matroids from \cite{mayhew2010stability} and the structure of $\{U_{2,5},U_{3,5}\}$-fragile matroids from \cite{clark2016fragile}.

In future work, we hope to obtain a description of these matroids that is independent of the notion of fragility. Specifically, we would like to characterize the representations of these matroids. As a step in this direction, we describe minimal $\gf(4)$-representations of matroids with a circuit-hyperplane whose relaxation is not $\gf(4)$-representable. Note that the proof uses the excluded-minor characterization of the class of  $\gf(4)$-representable matroids. The setup for this result is as follows.

Let $M$ be a $\gf(4)$-representable matroid on $E$ with a circuit-hyperplane $X$. Choose $e\in X$ and $f\in E-X$ such that $(X-e)\cup f$ is a basis of $M$. Then $M = M[I | C]$ for a quaternary matrix $C$ of the following block form.

\[C=\kbordermatrix{
&(E-X)-f  & e \\
X-e&A & \underline{1}\\
f &\underline{1}^{T} & 0\\
}.
\]

In the above matrix, $A$ is an $(X-e)\times ((E-X)-f)$ matrix, and we have scaled so that every non-zero entry in the row labelled by $f$ and the column labelled by $e$ is $1$. Let $M'$ be the matroid obtained from $M$ by relaxing the circuit-hyperplane $X$. We call the matrix $C$ a \emph{reduced representation} of $M$. If $M'$ is $\gf(4)$-representable, then we can find a reduced representation $C'$ of $M'$ in the following block form.

\[C'=\kbordermatrix{
&(E-X)-f  & e \\
X-e&A' & \underline{1}\\
f &\underline{1}^{T} & \omega\\
}.
\] 

We have scaled the rows and columns of the matrix such that the entry $C'_{fe} = \omega \in \gf(4)-\{0,1\}$, and the remaining entries in row $f$ and column $e$ are all $1$. The following theorem is our characterization in terms of forbidden submatrices.

\begin{theorem}
\label{badsubmatrix}
Let $M$ and $C$ be constructed as described above. There is a reduced representation $C'$ of the above form for $M'$ if and only if, up to permuting rows and columns, $A$ and $A^{T}$ have no submatrix in the following list, where $x,y,z$ denote distinct non-zero elements of $\gf(4)$.
\[\begin{bmatrix}
x&y&z
\end{bmatrix},
\begin{bmatrix}
x&y\\
0&x
\end{bmatrix}, 
\begin{bmatrix}
x&y\\
y&x
\end{bmatrix},
\begin{bmatrix}
x&x\\
y&z
\end{bmatrix},
\begin{bmatrix}
x&y\\
z&x
\end{bmatrix},
\begin{bmatrix}
x&x&0\\
x&0&x
\end{bmatrix},
\begin{bmatrix}
x&x&0\\
x&0&y
\end{bmatrix},
\]

\[\begin{bmatrix}
x&x&0\\
y&0&y
\end{bmatrix},
\begin{bmatrix}
x&y&0\\
x&0&y
\end{bmatrix},
\begin{bmatrix}
x&0&0\\
0&y&z
\end{bmatrix},
\begin{bmatrix}
x&y&0\\
x&0&z	
\end{bmatrix},
\begin{bmatrix}
x&0&0\\
0&x&0\\
0&0&x
\end{bmatrix},
\begin{bmatrix}
x&0&0\\
0&x&0\\
0&0&y
\end{bmatrix},\]

\[
\begin{bmatrix}
x&0&0\\
0&y&0\\
0&0&z
\end{bmatrix},
\begin{bmatrix}
x&y&x\\
y&y&0\\
x&0&0
\end{bmatrix},
\begin{bmatrix}
x&y&x\\
y&y&0\\
x&0&z
\end{bmatrix}.\]
\end{theorem}

This paper is organized as follows. In the next section, we collect some results on connectivity and circuit-hyperplane relaxation. In Section 3, we prove a fragility theorem. In Section 4, we describe the structure of the $\{U_{2,5},U_{3,5}\}$-fragile matroids. In Section 5, we prove the structural characterization. In Section 6, we reduce the proof of Theorem \ref{badsubmatrix} to a finite computer check. This check, carried out using SageMath, can be found in the Appendix \cite{COZarx}.

\section{Circuit-hyperplane relaxations and connectivity}
\label{prelims}

We assume the reader is familiar with the fundamentals of matroid theory. Any undefined matroid terminology will follow Oxley \cite{oxley2011matroid}. 
Let $M$ be a matroid on $E$, and let $\mathcal{B}(M)$ denote the collection of bases of $M$. If $M$ has a circuit-hyperplane $X$, then $\mathcal{B}(M')=\mathcal{B}(M)\cup \{X\}$ is the collection of bases of a matroid $M'$ on $E$. We say that $M'$ is obtained from $M$ by \textit{relaxing the circuit-hyperplane} $X$. We list here a number of useful results on circuit-hyperplane relaxation. 

\begin{lemma}
\cite[Proposition 2.1.7]{oxley2011matroid}
If $M'$ is obtained from $M$ by relaxing the circuit-hyperplane $X$ of $M$, then $(M')^{*}$ is obtained from $M^{*}$ by relaxing the circuit-hyperplane $E(M)-X$ of $M^{*}$.
\end{lemma}

The following elementary results are originally from \cite{kahn1985problem}.

\begin{lemma}
\label{Hminors}
 \cite[Proposition 3.3.5]{oxley2011matroid}
Let $X$ be a circuit-hyperplane of a matroid $M$, and let $M'$ be the matroid obtained from $M$ by relaxing $X$. When $e\in E(M)-X$,
\begin{enumerate}
\item[(i)] $M/e=M'/e$ and, unless $M$ has $e$ as a coloop, $M'\del e$ is obtained from $M\del e$ by relaxing the circuit-hyperplane $X$ of the latter.
\end{enumerate}
Dually, when $f\in X$,
\begin{enumerate}
\item[(ii)] $M\del f=M'\del f$ and, unless $M$ has $f$ as a loop, $M'/f$ is obtained from $M/f$ by relaxing the circuit-hyperplane $X-f$ of the latter.
\end{enumerate}
\end{lemma}

For a set $\mathcal{N}$ of matroids, we say that a matroid $M$ has an \textit{$\mathcal{N}$-minor} if $M$ has an $N$-minor for some $N\in \mathcal{N}$. We say $M$ is \textit{$\mathcal{N}$-fragile} if $M$ has an $\mathcal{N}$-minor and, for each element $e$ of $M$, at most one matroid in $\{M\del e, M/e\}$ has an $\mathcal{N}$-minor. We say an element $e$ of an $\mathcal{N}$-fragile matroid $M$ is \textit{nondeletable} if $M\del e$ has no $\mathcal{N}$-minor; the element $e$ is \textit{noncontractible} if $M/ e$ has no $\mathcal{N}$-minor.

The following lemma is an immediate consequence of Lemma \ref{Hminors}.

\begin{lemma}
\label{Hfragile}
Let $X$ be a circuit-hyperplane of a matroid $M$, and let $M'$ be the matroid obtained from $M$ by relaxing $X$. If $\mathcal{N}$ is a set of matroids such that $M'$ has an $\mathcal{N}$-minor but $M$ has no $\mathcal{N}$-minor, then $M'$ is $\mathcal{N}$-fragile. Moreover, $X$ is a basis of $M'$ whose elements are nondeletable such that the elements of the cobasis $E(M')-X$ are noncontractible.
\end{lemma}

We use the following connectivity result.

\begin{lemma}
 \cite[Proposition 8.4.2]{oxley2011matroid}
Let $M'$ be a matroid that is obtained by relaxing a circuit-hyperplane of a matroid $M$. If $M$ is $n$-connected, then $M'$ is $n$-connected.
\end{lemma}

Kahn \cite{kahn1985problem} proved the following result on the representability of a circuit-hyperplane relaxation.

\begin{lemma}
\label{nonbinary}
Let $M'$ be a matroid that is obtained by relaxing a circuit-hyperplane of a matroid $M$. If $M$ is connected, then $M'$ is non-binary.
\end{lemma}

We use the following definition of the rank function of the $2$-sum from \cite{jaeger1990computational}. Let $M_1$ and $M_2$ be matroids with at least two elements such that $E(M_1)\cap E(M_2)=\{p\}$. Then $M=M_1\oplus_2 M_2$ has rank function $r_M$ defined for all $A_1\subseteq E(M_1)$ and $A_2\subseteq E(M_2)$ by \[r_M(A_1\cup A_2)=r_{M_1}(A_1)+r_{M_2}(A_2)-\theta(A_1,A_2)+\theta(\emptyset, \emptyset)\]
where $\theta(X,Y)=1$ if $r_{M_1}(X\cup p)=r_{M_1}(X)$ and $r_{M_2}(Y\cup p)=r_{M_2}(Y)$, and $\theta(X,Y)=0$ otherwise.

The next three results on 2-sums and minors of $2$-sums are well known. 

\begin{lemma}
\label{2sumcircuits}
\cite[Proposition 7.1.20]{oxley2011matroid}
Let $M$ and $N$ be matroids with at least two elements. Let $E(M)\cap E(N)=\{p\}$ and suppose that neither $M$ nor $N$ has $\{p\}$ as a separator. The set of circuits of $M\oplus_2 N$ is
$$\mathcal{C}(M\del p)\cup \mathcal{C}(N\del p)\cup \{(C\cup D)-p: p\in C\in \mathcal{C}(M) \ and  \ p\in D\in \mathcal{C}(N)\}.$$
\end{lemma}

\begin{lemma}
\label{2sum}
\cite[Theorem 8.3.1]{oxley2011matroid}
A connected matroid $M$ is not $3$-connected if and only if $M=M_1\oplus_2 M_2$ for some matroids $M_1$ and $M_2$, each of which has at least three elements and is isomorphic to a proper minor of $M$.
\end{lemma}

\begin{lemma}
\label{2summinor}
 \cite[Proposition 8.3.5]{oxley2011matroid}
Let $M,N,M_1,M_2$ be matroids such that $M=M_1\oplus_2 M_2$ and $N$ is $3$-connected. If $M$ has an $N$-minor, then $M_1$ or $M_2$ has an $N$-minor.
\end{lemma}

We can now describe the structure of circuit-hyperplanes in matroids of low connectivity. We omit the straightforward proof of the next lemma.

\begin{lemma}
\label{disconnected}
Let $M$ be a $\gf(4)$-representable matroid with a circuit-hyperplane $H$. If $M$ is not connected, then $M\cong U_{1,m}\oplus U_{n-1,n}$ for some positive integers $m$ and $n$.
\end{lemma}

We now work towards a description of the $2$-separations of a connected matroid in which the relaxation of some circuit-hyperplane is $\gf(4)$-representable.

\begin{lemma}
\label{seriesparallel}
Let $M$ be a matroid with a circuit-hyperplane $X$. If $A$ is a non-trivial parallel class of $M$, then either $A\subseteq E-X$, or $A=X$ and $|A|=2$.
\end{lemma}

\begin{proof}
If $A\cap X$ and $A\cap (E-X)$ are both non-empty, then there is a circuit $\{x,y\}$ contained in $A$ such that $x\in X$ and $y\in E-X$. But $E-X$ is a cocircuit of $M$, so this is a contradiction to orthogonality. Thus either $A\cap X$ or $A\cap (E-X)$ is empty. In the case that $A\cap (E-X)$ is empty, there is a circuit $\{x,y\}$ contained in $A$ that is also contained in the circuit $X$, so $X=A=\{x,y\}$.
\end{proof}

For the next result, we say that $M$ is \textit{$3$-connected up to series and parallel classes} if $M$ is connected and, for any $2$-separation $(X,Y)$ of $M$, either $X$ or $Y$ is a series class or a parallel class.

\begin{lemma}
\label{nt2sep}
Let $M$ be a $\gf(4)$-representable matroid with a circuit-hyperplane $X$ such that the matroid $M'$ obtained from $M$ is also $\gf(4)$-representable. If $M$ is connected but not $3$-connected, then $M$ is $3$-connected up to series and parallel classes. 
\end{lemma}

\begin{proof}
Assume that $M$ has a $2$-separation $(S,T)$ where neither side is a series or parallel class. Then $M$ has a $2$-sum decomposition of the form $M=N\oplus_2 N'$ for some $N$ and $N'$ with $E(N)\cap E(N')=\{p\}$, where neither $N$ nor $N'$ is a circuit or cocircuit. 

First suppose that the circuit $X$ of $M$ has the form $(C\cup C')-p$, where $C$ is a circuit of $N$, and $C'$ is a circuit of $N'$ while $p \in C\cap C'$. Then 

\begin{equation}
\label{2sep1}
r(X)=r(M)-1,
\end{equation} 
\begin{equation}
\label{2sep2}
r(N)+r(N')-1=r(M),
\end{equation}and 
\begin{equation}
\label{2sep3}
r_M(X)= r_N(C)+r_{N'}(C')-1.
\end{equation}  

Equation \eqref{2sep1} follows from the fact that $X$ is a hyperplane of $M$; Equations \eqref{2sep2} and \eqref{2sep3} follow from the definition of the rank function of the $2$-sum of $N$ and $N'$. Combining \eqref{2sep1} and \eqref{2sep2}, we see that $r(X)=r(N)+r(N')-2$. Then combining this equation with \eqref{2sep3}, we see that \[r(C)+r(C')=r(N)+r(N')-1.\] We may therefore assume that $C$ is a spanning circuit of $N$, and hence that $E(N)=C$ because the hyperplane $X$ is closed. Therefore $N$ is a circuit, a contradiction.

By symmetry, it remains to consider the case when $X$ is a circuit of $N'\del p$. Then $r(X)\leq r(N')$. Since $X$ is a hyperplane of $M$, and $r(M)=r(N)+r(N')-1$, it follows that $r(N)\leq 2$. Since $N$ is not a cocircuit, we deduce that $r(N)=2$. Then $r(M)=r(N')+1$, so $r(X)=r(N')=r(N'\del p)$. Since $N$ is not a circuit we deduce that $\si(N)\cong U_{2,m}$ for some $m\geq 4$. Moreover, $p$ is not in a non-trivial parallel class in $N$ otherwise $X$ is not a hyperplane of $M$.

Switching to $M^{*}$, we see that $r_{M^{*}}(N') = |X| + r(N) - r(M) = r(N) = 2$. As above, it follows that $\co(N')\cong U_{n-1,n+1}$ for some $n\geq 3$. Moreover, $p$ is not in a non-trivial series class in $N'$. Let $X_1$ consist of one representative of each series class of $N'$, and let $Y_1$ consist of one representative of each parallel class of $N$. By contracting elements of $X-X_1$ and deleting elements of $(E(M)-X)-Y_1$, we obtain  $U_{n-1,n+1}\oplus_2 U_{2,m}$ as a minor of $M$ for some $n\geq 3$ and $m\geq 4$. Moreover, by Lemma \ref{Hminors}, $X_1$ is a circuit-hyperplane of this minor whose relaxation is $\gf(4)$-representable. Thus $X_1\subseteq E(U_{n-1,n+1})$. Contract $n-3$ elements from $X_1$ and delete $m-4$ elements from $Y_1$ to get $U_{2,4}\oplus_2 U_{2,4}$. Relaxing a circuit-hyperplane of this minor gives $P_6$ which is not $\gf(4)$-representable (see \cite[Proposition 6.5.8]{oxley2011matroid}), a contradiction.\end{proof}

\section{A fragility theorem}

We will use the following consequence of Geelen, Oxley, Vertigan, and Whittle \cite[Theorem 8.4]{geelen1998weak}. 

\begin{theorem}
\label{geelen1998weak}
Let $M$ and $M'$ be $\gf(4)$-representable matroids with the properties that $M$ is connected, $M'$ is $3$-connected, and $M'$ is obtained from $M$ by relaxing a circuit-hyperplane.
\begin{enumerate}
\item[(i)] If $M'$ has a $U_{2,4}$-minor but no $\{U_{2,5}, U_{3,5}\}$-minor, then $M$ is binary.
\item[(ii)] If $M'$ has a $\{U_{2,5}, U_{3,5}\}$-minor but no $U_{3,6}$-minor, then $M$ has no $\{U_{2,5}, U_{3,5}\}$-minor.
\end{enumerate}
\end{theorem}

We can now prove the main result of this section.

\begin{theorem}
\label{fragilitymain}
Let $M$ and $M'$ be $\gf(4)$-representable matroids such that $M$ is connected, $M'$ is $3$-connected, and $M'$ is obtained from $M$ by relaxing a circuit-hyperplane $X$. Then $M'$ is either $U_{2,4}$-fragile or $\{U_{2,5}, U_{3,5}\}$-fragile. Moreover, $X$ is a basis of $M'$ whose elements are nondeletable such that the elements of the cobasis $E(M')-X$ are noncontractible.
\end{theorem}

\begin{proof}
First assume that $M'$ has no $\{U_{2,5}, U_{3,5}\}$-minor. By Lemma \ref{nonbinary} and Theorem \ref{geelen1998weak} (i), $M'$ has a $U_{2,4}$-minor and $M$ has no $U_{2,4}$-minor. Then it follows from Lemma \ref{Hfragile} that $M'$ is $U_{2,4}$-fragile, and $M'$ has a basis $X$ whose elements are nondeletable such that the elements of the cobasis $E(M')-X$ are noncontractible.

We may now assume that $M'$ has a $\{U_{2,5}, U_{3,5}\}$-minor. Suppose that $M$ also has a $\{U_{2,5}, U_{3,5}\}$-minor, and assume that $M$ is a minor-minimal matroid with respect to the hypotheses; that is, we assume that $M$ has no proper minor $M_0$ such that $M_0$ is connected, $M_0$ has a $\{U_{2,5}, U_{3,5}\}$-minor, and $M_0$ has a circuit-hyperplane whose relaxation $M_0'$ is $3$-connected, $\gf(4)$-representable, and has a $\{U_{2,5}, U_{3,5}\}$-minor.

\begin{claim}
\label{minfragile}
$M$ is $\{U_{2,5}, U_{3,5}\}$-fragile.
\end{claim}

\begin{subproof}
Suppose that $M$ has an element $e\in E(M)-X$ such that $M\del e$ has a $\{U_{2,5}, U_{3,5}\}$-minor. If $M\del e$ is $3$-connected, then we have a contradiction to the minimality of $M$. Therefore, by Lemma \ref{nt2sep}, $M\del e$ is $3$-connected up to series and parallel pairs. Suppose that $A$ is a non-trivial parallel class of $M\del e$. Suppose $A\subseteq X$. Then $A=X$ and $|A|=2$ by Lemma \ref{seriesparallel}, so we deduce that $M\del e$ is a parallel extension of $U_{2,5}$ and hence that $M'\del e$ has a $U_{2,6}$-minor, a contradiction to the fact that the matroid $M'$ obtained from $M$ by relaxing $X$ is $\gf(4)$-representable. Thus $A\subseteq E(M\del e)-X$ by Lemma \ref{seriesparallel}. By duality, any non-trivial series class of $M\del e$ must be contained in $X$. Then, by Lemma \ref{2summinor}, the matroid $M_0$ obtained from $M\del e$ by deleting all but one element of every non-trivial parallel class and contracting all but one element of every non-trivial series class has a $\{U_{2,5}, U_{3,5}\}$-minor. We deduce from Lemma \ref{nt2sep} that $M_0$ is $3$-connected. Then $M_0$ contradicts the minimality of $M$. Therefore $M\del e$ has no $\{U_{2,5}, U_{3,5}\}$-minor for all $e\in E(M)-X$, and, by duality, $M/e$ has no $\{U_{2,5}, U_{3,5}\}$-minor for all $e\in X$, so $M$ is $\{U_{2,5}, U_{3,5}\}$-fragile. This completes the proof of \ref{minfragile}. \phantom\qedhere
\end{subproof}

Since $M$ has a $\{U_{2,5}, U_{3,5}\}$-minor, it follows from Theorem \ref{geelen1998weak} (ii) that $M'$ has a $U_{3,6}$-minor, that is, $M'/C\del D\cong U_{3,6}$ for some subsets $C$ and $D$. If $C\subseteq X$ and $D\subseteq E(M')-X$, then it follows from Lemma \ref{Hminors} that $U_{3,6}$ can be obtained from $M/C\del D$ by relaxing the circuit-hyperplane $X-C$. Hence $M/C\del D\cong P_6$, a contradiction because $M/C\del D$ is $\gf(4)$-representable but $P_6$ is not. Therefore $C\cap (E(M')-X)$ or $D\cap X$ is nonempty, so $M/C\del D=M'/C\del D\cong U_{3,6}$ by Lemma \ref{Hminors}. This is a contradiction to \ref{minfragile} because any minor of $M$ must also be $\{U_{2,5}, U_{3,5}\}$-fragile, but for any $e$, both $U_{3,6}\del e$ and $U_{3,6}/e$ have a $\{U_{2,5}, U_{3,5}\}$-minor. We conclude that $M$ has no $\{U_{2,5}, U_{3,5}\}$-minor. It now follows from Lemma \ref{Hfragile} that $M'$ is $\{U_{2,5}, U_{3,5}\}$-fragile, and that $M'$ has a basis $X$ whose elements are nondeletable such that the elements of the cobasis $E(M')-X$ are noncontractible.
\end{proof}

\section{The structure of $\{U_{2,5},U_{3,5}\}$-fragile matroids}
\label{fragilestructure}

\subsection{Partial Fields and Constructions}

We briefly state the necessary material on partial fields. For a more thorough treatment, we refer the reader to \cite{pendavingh2010confinement}.

A \emph{partial field} is a pair $\mathbb{P} = (R,G)$, where $R$ is a commutative ring with unity, and $G$ is a subgroup of the units of $R$ with $-1 \in G$. A matrix with entries in $G$ is a \emph{$\mathbb{P}$-matrix} if $\det(D) \in G \cup \{0\}$ for any square submatrix $D$ of $A$. We use $\langle X \rangle$ to denote the multiplicative subgroup of $R$ generated by the subset $X$.

A rank-$r$ matroid $M$ on the ground set $E$ is \emph{$\mathbb{P}$-representable} if there is an $r\times |E|$ $\mathbb{P}$-matrix $A$ such that, for each $r\times r$ submatrix $D$, the determinant of $D$ is nonzero if and only if the corresponding subset of $E$ is a basis of $M$. When this occurs, we write $M = M[A]$. 

The \textit{$2$-regular} partial field is defined as follows. 

\[\mathbb{U}_2 = \left(\mathbb{Q}(\alpha,\beta), \langle -1, \alpha, \beta, 1-\alpha, 1-\beta,\alpha-\beta\rangle\right),\]
where $\alpha$,$\beta$ are indeterminates.

It is well-known that any $\mathbb{U}_2$-representable matroid is $\gf(4)$-representable \cite{oxley2000generalized}. On the other hand, there are $\gf(4)$-representable matroids that are not $\mathbb{U}_2$-representable. We now define three such matroids. The matroid $P_8$ has a unique pair of disjoint circuit-hyperplanes; we let $P_8^{-}$ denote the unique matroid obtained by relaxing one of these circuit-hyperplanes. We denote by $F_7^{=}$ the matroid obtained from the non-Fano matroid $F_7^{-}$ by relaxing a circuit-hyperplane. The $\gf(4)$-representable matroids $P_8^{-}$,$F_7^{=}$, $(F_7^{=})^{*}$ are not $\mathbb{U}_2$-representable. We note that this can be deduced from \cite{chunfragile} since $P_8^{-}$,$F_7^{=}$, $(F_7^{=})^{*}$ are $\{U_{2,5}, U_{3,5}\}$-fragile matroids. Since these matroids are not $\mathbb{U}_2$-representable, we have the following lemma.

\begin{lemma}
\label{not2regular}
The class of $\mathbb{U}_2$-representable matroids is contained in the class of $\gf(4)$-representable matroids with no $\{P_8^{-},F_7^{=}, (F_7^{=})^{*}\}$-minor.
\end{lemma}

To describe the structure of $\{U_{2,5}, U_{3,5}\}$-fragile matroids as in \cite{clark2016fragile}, we need two constructions: the generalized $\Delta$-$Y$ exchange, and gluing on wheels. For a more thorough treatment of these constructions, we refer the reader to \cite{oxley2000generalized} and \cite{chun2013fan}.

Loosely speaking, the operations of generalized $\Delta$-$Y$ exchange and gluing on wheels both involve gluing matroids together along a common restriction. Let $M_1$ and $M_2$ be matroids with a common restriction $A$, where $A$ is a modular flat of $M_1$. The \textit{generalized parallel connection} of $M_1$ and $M_2$ along $A$, denoted $P_A(M_1,M_2)$, is the matroid obtained by gluing $M_1$ and $M_2$ along $A$. It has ground set $E(M_1)\cup E(M_2)$, and a set $F$ is a flat of $P_A(M_1,M_2)$ if and only if $F\cap E(M_i)$ is a flat of $M_i$ for each $i$ (see \cite[Section 11.4]{oxley2011matroid}). 

A subset $S$ of $E(M)$ is a \textit{segment} of $M$ if every three-element subset of $S$ is a triangle of $M$. Let $M$ be a matroid with a $k$-element segment $A$. Intuitively, a generalized $\Delta$-$Y$ exchange on $A$ turns the segment $A$ into a $k$-element cosegment. To define the generalized $\Delta$-$Y$ exchange formally, we first recall the following definition of a family of matroids $\Theta_k$ from \cite{oxley2000generalized}. For $k\geq 3$, fix a basis $B=\{b_1,b_2,\ldots, b_k\}$ of the rank-$k$ projective geometry $PG(k-1,\mathbb{R})$, and choose a line $L$ of $PG(k-1,\mathbb{R})$ that is freely placed relative to $B$. If follows from modularity that, for each $i$, the hyperplane spanned by $B-\{b_i\}$ meets $L$; we let $a_i$ be the point of intersection. Let $A=\{a_1,a_2,\ldots, a_k\}$, and let $\Theta_k$ be the matroid obtained by restricting $PG(k-1,\mathbb{R})$ to the set $A\cup B$. Note that the matroid $\Theta_k$ has $A$ as a modular $k$-point segment $A$, so the generalized parallel connection of $\Theta_k$ and $M$ along $A$ is well-defined. If the $k$-element segment $A$ is coindependent in $M$, then we define the matroid $\Delta_A(M)$ to be the matroid obtained from $P_A(\Theta_k,M)\del A$ by relabeling the elements of $E(\Theta_k) - A$ by $A$ in the natural way, and we say that $\Delta_A(M)$ is obtained from $M$ by performing a \textit{generalized $\Delta$-$Y$ exchange} on $A$. For a matroid $M$ with an independent cosegment $A$, a \textit{generalized $Y$-$\Delta$ exchange on $A$}, denoted by $\nabla_A(M)$, is defined to be the matroid $(\Delta_A(M^{*}))^{*}$. 

We use the following results on representability and the minor operations.

\begin{lemma}
\label{3.7}
\cite[Lemma 3.7]{oxley2000generalized}
Let $\mathbb{P}$ be a partial field. Then $M$ is $\mathbb{P}$-representable if and only if $\Delta_A(M)$ is $\mathbb{P}$-representable.
\end{lemma}

\begin{lemma}
\label{2.13}
\cite[Lemma 2.13]{oxley2000generalized}
Suppose that $\Delta_A(M)$ is defined. If $x\in A$ and $|A|\geq 3$, then $\Delta_{A-x}(M\del x)$ is also defined, and $\Delta_A(M)/x=\Delta_{A-x}(M\del x)$.
\end{lemma}

\begin{lemma}
\label{2.16}
\cite[Lemma 2.16]{oxley2000generalized}
Suppose that $\Delta_A(M)$ is defined.
\begin{enumerate}
\item[(i)] If $x\in E(M)-A$ and $A$ is coindependent in $M\del x$, then $\Delta_A(M\del x)$ is defined and $\Delta_A(M)\del x=\Delta_A(M\del x)$.
\item[(ii)] If $x\in E(M)-\cl(A)$, then $\Delta_A(M/x)$ is defined and $\Delta_A(M)/x=\Delta_A(M/x)$.
\end{enumerate}
\end{lemma}  

\begin{lemma}
\label{2.15}
\cite[Lemma 2.15]{oxley2000generalized}
Suppose that  $x\in \cl(A)-A$ and let $a$ be an arbitrary element of the $k$-element segment $A$. Then $\Delta_A(M)/x$ equals the $2$-sum, with basepoint $p$, of a copy of $U_{k-1,k+1}$ with groundset $A\cup p$ and the matroid obtained from $M/x\del (A-a)$ by relabeling $a$ as $p$.
\end{lemma} 

The next result implies that every $\{U_{2,5}, U_{3,5}\}$-fragile matroid is $3$-connected up to series and parallel classes. 

\begin{lemma}
\label{fcon}
\cite[Proposition 4.3]{mayhew2010stability}
Let $M$ be a matroid with a $2$-separation $(A,B)$, and let $N$ be a $3$-connected minor of $M$. Assume $|E(N)\cap A|\geq |E(N)\cap B|$. Then $|E(N)\cap B|\leq 1$. Moreover, unless $B$ is a parallel or series class, there is an element $x\in B$ such that both $M\del x$ and $M/x$ have a minor isomorphic to $N$. 
\end{lemma}

The following is an easy consequence of the property that $\{U_{2,5}, U_{3,5}\}$-fragile matroids are $3$-connected up to parallel and series classes. 

\begin{lemma}
\label{incotri}
Let $M$ be a $\{U_{2,5}, U_{3,5}\}$-fragile matroid with at least $8$ elements. If $S$ is a triangle or $4$-element segment of $M$ such that $E(M)-S$ is not a series or parallel class of $M$, then $S$ is coindependent in $M$. If $C$ is a triad or $4$-element cosegment of $M$ such that $E(M)-C$ is not a series or parallel class of $M$, then $C$ is independent. 
\end{lemma}

Let $M$ be a $\{U_{2,5},U_{3,5}\}$-fragile matroid.  A segment $S$ of $M$ is \textit{allowable} if $S$ is coindependent and some element of $S$ is nondeletable. A cosegment $C$ of $M$ is \textit{allowable} if the segment $C$ of $M^{*}$ is allowable. In \cite{clark2016fragile}, it was shown that we can obtain a new $\{U_{2,5}, U_{3,5}\}$-fragile $\mathbb{U}_2$-representable matroid from an old $\{U_{2,5}, U_{3,5}\}$-fragile $\mathbb{U}_2$-representable matroid by performing a generalized $\Delta$-$Y$ exchange on an allowable segment. We will prove an analogous result for $\{U_{2,5}, U_{3,5}\}$-fragile $\gf(4)$-representable matroids with no $\{P_8^{-},F_7^{=}, (F_7^{=})^{*}\}$-minor. 

Let $\mathcal{U}$ be the class of $\gf(4)$-representable matroids with no $\{U_{2,5}, U_{3,5}\}$-minor. The class of \textit{sixth-root-of-unity} matroids is the class of matroids that are representable over both $\gf(3)$ and $\gf(4)$. Semple and Whittle \cite[Theorem 5.2]{semple1996representable} showed that $\mathcal{U}$ is the class of matroids that can be obtained by taking direct sums and $2$-sums of binary and sixth-root-of-unity matroids.

\begin{lemma}
\label{nou25oru35}
Let $M$ be a matroid in the class $\mathcal{U}$. If $M'$ is obtained from $M$ by performing a generalized $\Delta$-$Y$ exchange or a generalized $Y$-$\Delta$ exchange, then $M'\in \mathcal{U}$.
\end{lemma}

\begin{proof}
Suppose that there exists a matroid $M\in \mathcal{U}$ with a coindependent segment $A$ such that $\Delta_A(M) \not\in \mathcal{U}$. Among all counterexamples, suppose that $M$ has been chosen so that $|E(M)|$ is as small as possible. Suppose $M$ is $3$-connected. Since any $3$-connected member of $\mathcal{U}$ is either a binary or sixth-root-of-unity matroid, this also holds for $\Delta_A(M)$ by Lemma \ref{3.7}. Hence $\Delta_A(M)\in \mathcal{U}$, contradicting the assumption that $M$ is a counterexample. Therefore $M$ is not $3$-connected. 

Now either $M=M_1\oplus M_2$ or $M=M_1\oplus_2 M_2$ for some $M_1, M_2\in \mathcal{U}$ with $|E(M_i)|<|E(M)|$ for each $i\in \{1,2\}$. Moreover, we may assume that $M_1$ and $M_2$ have been chosen so that the segment $A$ of $M$ is contained in $E(M_1)$. Now either $\Delta_A(M)=\Delta_A(M_1)\oplus M_2$ or $\Delta_A(M)=\Delta_A(M_1)\oplus_2 M_2$. Since $|E(M_1)|<|E(M)|$, it follows that $\Delta_A(M_1)\in \mathcal{U}$. Hence $\Delta_A(M)\in \mathcal{U}$. Since $\mathcal{U}$ is closed under duality, the result follows.
\end{proof}

\begin{lemma}
\label{amovegf4} 
Let $M$ be a $\{U_{2,5}, U_{3,5}\}$-fragile $\gf(4)$-representable matroid with no $\{P_8^{-},F_7^{=},$ $(F_7^{=})^{*}\}$-minor. If $A$ is an allowable segment of $M$ with $|A| \in \{3,4\}$, then $\Delta_A(M)$ is a  $\{U_{2,5}, U_{3,5}\}$-fragile $\gf(4)$-representable matroid with no $\{P_8^{-},F_7^{=}, (F_7^{=})^{*}\}$-minor. Moreover, $A$ is an allowable cosegment of $\Delta_A(M)$. 
\end{lemma}

\begin{proof}
The proof that $\Delta_A(M)$ is a  $\{U_{2,5}, U_{3,5}\}$-fragile $\gf(4)$-representable matroid where $A$ is an allowable cosegment of $\Delta_A(M)$ closely follows the proof of \cite[Lemma 2.21]{clark2016fragile}. The only difference is where the proof of \cite[Lemma 2.21]{clark2016fragile} uses the fact that a $\mathbb{U}_2$-representable matroid with no $\{U_{2,5}, U_{3,5}\}$-minor is near-regular and the class of near-regular matroids is closed under the generalized $\Delta$-$Y$ exchange, we instead use Lemma \ref{nou25oru35}.

We must also show that $\Delta_A(M)$ has no $\{P_8^{-},F_7^{=}, (F_7^{=})^{*}\}$-minor. This follows for $|E(M)|\leq 9$ from the generation of the $3$-connected $\{U_{2,5}, U_{3,5}\}$-fragile $\gf(4)$-represen-table matroids with no $\{P_8^{-},F_7^{=}, (F_7^{=})^{*}\}$-minor on at most $9$ elements (see the Appendix \cite{COZarx}), since all such matroids are $\mathbb{U}_2$-representable. Suppose that $M$ is a minimum-sized counterexample, so $\Delta_A(M)$ has a $\{P_8^{-},F_7^{=}, (F_7^{=})^{*}\}$-minor and $\Delta_A(M)$ has at least ten elements. Then $\Delta_A(M)$ has a minor $N$, obtained by deleting or contracting an element $x$ say, that also has a $\{P_8^{-},F_7^{=}, (F_7^{=})^{*}\}$-minor. Since $\Delta_A(M)$ is $\{U_{2,5}, U_{3,5}\}$-fragile it follows that the minor $N$ is also $\{U_{2,5}, U_{3,5}\}$-fragile. Suppose that $N= \Delta_A(M)/x$. Suppose that $x\in A$. Then $\Delta_A(M)/x=\Delta_{A-x}(M\del x)$ by Lemma \ref{2.13}, a contradiction since $M$ is a minimum-sized counterexample. Next suppose that $x\in \cl(A)-A$. Since $N$ is $\{U_{2,5}, U_{3,5}\}$-fragile it follows from Lemma \ref{2.15} and Proposition \ref{fcon} that $|A|=4$ and $M/x\del (A-a)\cong U_{1,n}$ for some $n\geq 2$. Hence $\Delta_A(M)$ has no $\{P_8^{-},F_7^{=}, (F_7^{=})^{*}\}$-minor, a contradiction. We may now assume $x\in E(M)-\cl(A)$. Then $\Delta_A(M)/x=\Delta_A(M/x)$ by Lemma \ref{2.16}, a contradiction since $M$ is a minimum-sized counterexample. We deduce that $N=\Delta_A(M)\del x$, and we may assume that any minor obtained from $\Delta_A(M)$ by contracting an element has no $\{P_8^{-},F_7^{=}, (F_7^{=})^{*}\}$-minor. Now if $x\in A$, then $A-x$ is a series class of $\Delta_A(M)\del x$, so there is some $y\in A$ such that $\Delta_A(M)/y$ has a $\{P_8^{-},F_7^{=}, (F_7^{=})^{*}\}$-minor, a contradiction. Therefore $x\notin A$. If $A$ is not coindependent in $M\del x$, then it follows from Lemma \ref{incotri} that $\Delta_A(M)$ has no $\{P_8^{-},F_7^{=}, (F_7^{=})^{*}\}$-minor, a contradiction. Therefore $A$ is coindependent in $M\del x$, so $\Delta_A(M)\del x=\Delta_A(M\del x)$ by Lemma \ref{2.16}, a contradiction since $M$ is a minimum-sized counterexample.
\end{proof}

Let $M$ be a matroid, and $(a,b,c)$ an ordered subset of $E(M)$ such that $T = \{a,b,c\}$ is a triangle. Let $r\geq 3$ be a positive integer, and, when $r=3$, we fix a vertex of $\mathcal{W}_3$ to be the center, so we can refer to rim and spoke elements of $M(\mathcal{W}_3)$. Let $N$ be obtained from $M(\mathcal{W}_r)$ by relabeling some triangle as $\{a,b,c\}$, where $a,c$ are spoke elements, and let $X \subseteq \{a,b,c\}$ such that $b \in X$. We say the matroid $M' := P_T(M,N) \delete X$ is obtained from $M$ by \emph{gluing an $r$-wheel onto $(a,b,c)$}. We also say that $M^{*}$ is obtained from $N^{*}$ by gluing a wheel onto the triad $T$. Suppose that $T_1, T_2, \ldots, T_n$ are ordered triples whose underlying sets are triangles of $M$. We say $M'$ can be obtained from $M$ by \textit{gluing wheels onto $T_1, T_2, \ldots, T_n$} if, for some subset $J$ of $\{1,2,\ldots, n\}$, $M'$ can be obtained from $M$ by a sequence of moves, where each move consists of gluing an $r_j$-wheel onto $T_j$ for $j\in J$. Note that the spoke elements of a triangle in this sequence may only be deleted as part of the gluing operation when they do not appear in any subsequent triangle in the sequence.

\begin{lemma}
\label{fanfragilegf4}
Let $M$ be a $\{U_{2,5}, U_{3,5}\}$-fragile $\gf(4)$-representable matroid with no $\{P_8^{-},F_7^{=},$ $(F_7^{=})^{*}\}$-minor. Let $A = \{a,b,c\}$ be an allowable triangle of $M$, where $b$ is nondeletable. If $M'$ is obtained from $M$ by gluing an $r$-wheel onto $(a,b,c)$, where $X \subseteq \{a,b,c\}$ is such that $b \in X$, then $M'$ is a $\{U_{2,5}, U_{3,5}\}$-fragile $\gf(4)$-representable matroid with no $\{P_8^{-},F_7^{=}, (F_7^{=})^{*}\}$-minor. Moreover, $F = E(\mathcal{W}_r) - X$ is the set of elements of a fan, the spoke elements of $F$ are noncontractible in $M'$, and the rim elements of $F$ are nondeletable in $M'$. 
\end{lemma}

\begin{proof}
The proof is the same as \cite[Lemma 2.22]{clark2016fragile} except that we use Lemma \ref{amovegf4} instead of \cite[Lemma 2.21]{clark2016fragile}.
\end{proof}

\subsection{Path sequences}

We can now describe a family of $\{U_{2,5}, U_{3,5}\}$-fragile $\gf(4)$-representable matroids with no $\{P_8^{-},F_7^{=}, (F_7^{=})^{*}\}$-minor obtained by performing generalized $\Delta$-$Y$ exchanges and gluing on wheels. In fact, the matroids in this family are $\mathbb{U}_2$-representable and were first described in \cite{clark2016fragile}. Each matroid in this family has a $\{X_8, Y_8, Y_8^*\}$-minor, and an associated path of $3$-separations that we need to describe in order to define the family.

We call the set $X \subseteq E(M)$ \textit{fully closed} if $X$ is closed in both $M^{*}$ and $M$. The \textit{full closure} of $X$, denoted $\fcl_M(X)$, is the intersection of all fully closed sets containing $X$. The full closure of $X$ can be obtained from $X$ by repeatedly taking closure and coclosure until no new elements are added. We call $X$ a \textit{path-generating} set if $X$ is a $3$-separating set of $M$ such that $\fcl_M(X)=E(M)$. A path-generating set $X$ thus gives rise to a natural path of $3$-separating sets $(P_1,\ldots, P_m)$, where $P_1 = X$ and each step $P_i$ is either the closure or coclosure of the $3$-separating set $P_1\cup \cdots \cup P_{i-1}$. 

Let $X$ be an allowable cosegment of the $\{U_{2,5},U_{3,5}\}$-fragile matroid $M$. A matroid $Q$ is an \textit{allowable series extension of $M$ along $X$} if $M=Q/Z$ and, for every element $z$ of $Z$, there is some element $x$ of $X$ such that $x$ is $\{U_{2,5},U_{3,5}\}$-contractible in $M$ and $z$ is in series with $x$ in $Q$. We also say that $Q^{*}$ is an \textit{allowable parallel extension of $M^{*}$ along $X$}.

Let $N$ be a matroid with a path-generating allowable segment or cosegment $A$. We say that $M$ is obtained from $N$ by a \textit{\dy-step} along $A$ if, up to duality, $M$ is obtained from $N$ by performing a non-empty allowable parallel extension along $A$, followed by a generalized $\Delta$-$Y$ exchange on $A$. 

Let $X_8$ be the matroid obtained from $U_{2,5}$ by choosing a $4$-element segment $C$, adding a point in parallel with each of three distinct points of $C$, then performing a generalized $\Delta$-$Y$-exchange on $C$ (see Figure \ref{fig:Y8X8}). In what follows, $S$ will be the elements of the 4-element segment of $X_8$, and $C$ the elements of the 4-element cosegment of $X_8$, so $E(X_8) = S \cup C$. We will build matroids from $X_8$ by performing a sequence of \dy-steps along $A \in \{S,C\}$. Note that, in such matroids, each of $S$ and $C$ can be either a segment or a cosegment.

\begin{figure}[tbp]
        \centering
        \includegraphics{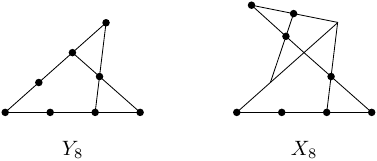}
        
        \caption{The matroids $Y_8$ and $X_8$.}\label{fig:Y8X8}
\end{figure}

A sequence of matroids $M_1,\ldots, M_n$ is called a \textit{path sequence} if the following conditions hold:
\begin{enumerate}
 \item[(P1)] $M_1=X_8$; and
 \item[(P2)] For each $i\in \{1,\ldots, n-1\}$, there is some $4$-element path-generating segment or cosegment $A \in \{S,C\}$ of $M_i$ such that either: 
\begin{itemize}
 \item[(a)] $M_{i+1}$ is obtained from $M_{i}$ by a \dy-step along $A$; or
 \item[(b)] $M_{i+1}$ is obtained from $M_i$ by gluing a wheel onto an allowable subset $A'$ of $A$.
\end{itemize}
\end{enumerate}

Note in (P2) that each \dy-step described in (a) increases the number of elements by at least one, and that the wheels in (b) are only glued onto allowable subsets of 4-element segments or cosegments.

We say that a path sequence $M_1,\ldots, M_n$ \textit{describes} a matroid $M$ if $M_n\cong M$. We also say that $M$ is a matroid \textit{described by} a path sequence if there is some path sequence that describes $M$. Let $\mathcal{P}$ denote the class of matroids such that $M\in \mathcal{P}$ if and only if there is some path sequence $M_1,\ldots, M_n$ that describes a matroid $M'$ such that $M$ can be obtained from $M'$ by some, possibly empty, sequence of allowable parallel and series extensions. Since $X_8$ is self-dual, it is easy to see that the sequence of dual matroids $M_1^{*},\ldots, M_n^{*}$ of a path sequence $M_1,\ldots, M_n$ is also a path sequence.  Thus the class $\mathcal{P}$ is closed under duality. 

We denote by $Y_8$ the unique matroid obtained from $X_8$ by performing a $Y$-$\Delta$-exchange on an allowable triad (see Figure \ref{fig:Y8X8}). We will prove the following result. 

\begin{theorem}
\label{mainresultgf4}
If $M$ is a $3$-connected $\{U_{2,5}, U_{3,5}\}$-fragile $\gf(4)$-representable matroid that has an $\{X_8,Y_8,Y_8^{*}\}$-minor but no $\{P_8^{-},F_7^{=}, (F_7^{=})^{*}\}$-minor, then there is some path sequence that describes $M$.
\end{theorem}

The proof of Theorem \ref{mainresultgf4} closely follows the proof of \cite[Corollary 4.3]{clark2016fragile}. The strategy is to show that a minor-minimal counterexample has at most 12 elements. Let $M$ be a $\gf(4)$-representable $\{U_{2,5}, U_{3,5}\}$-fragile matroid $M$ with an $\{X_8,Y_8,Y_8^{*}\}$-minor but no $\{P_8^{-},F_7^{=}, (F_7^{=})^{*}\}$-minor. Suppose that $M$ is a minimum-sized matroid that is not in the class $\mathcal{P}$. Then $M$ is $3$-connected because $\mathcal{P}$ is closed under series and parallel extensions. Moreover, the dual $M^{*}$ is also not in $\mathcal{P}$ because $\mathcal{P}$ is closed under duality. Thus, by the Splitter Theorem and duality, we may assume there is some element $x$ of $M$ such that $M\del x$ is also a $3$-connected $\gf(4)$-representable $\{U_{2,5}, U_{3,5}\}$-fragile matroid with an $\{X_8,Y_8,Y_8^{*}\}$-minor but no $\{P_8^{-},F_7^{=}, (F_7^{=})^{*}\}$-minor. By the assumption that $M$ is minimum-sized with respect to being outside the class $\mathcal{P}$, it follows that $M\del x\in \mathcal{P}$. Thus $M\del x$ is described by a path sequence $M_1,\ldots, M_n$. The next lemma \cite[Lemma 6.3]{clark2016fragile} identifies the three possibilities for the position of $x$ in $M$ relative to the path of $3$-separations associated with $M_1,\ldots, M_n$. 

\begin{lemma}
\label{blocksorguts}
Let $M$ and $M\del x$ be $3$-connected $\{U_{2,5}, U_{3,5}\}$-fragile matroids. If $M\del x$ is described by a path sequence with associated path of $3$-separations $\mathbf{P}$, then either:
\begin{enumerate}
 \item[(i)] there is some $3$-separation $(X,Y)$ displayed by $\mathbf{P}$ such that $x\in \cl(X)$ and $x\in \cl(Y)$; or
 \item[(ii)] there is some $3$-separation $(X,Y)$ displayed by $\mathbf{P}$ such that $x\notin \cl(X)$ and $x\notin \cl(Y)$; or 
 \item[(iii)] for each $3$-separation $(R,G)$ of $M$ displayed by $\mathbf{P}$, there is some $X\in \{R,G\}$ such that $x\in \cl_M(X)$ and $x\in \cl_M^{*}(X)$.
\end{enumerate}
\end{lemma}

The proofs of the next three lemmas follow the proofs of \cite[Lemma 7.4]{clark2016fragile}, \cite[Lemma 8.7]{clark2016fragile}, and \cite[Lemma 9.7]{clark2016fragile} but use Lemma \ref{amovegf4} above instead of \cite[Lemma 2.21]{clark2016fragile}.

\begin{lemma}
Lemma \ref{blocksorguts} (i) does not hold. 
\end{lemma}

\begin{lemma}
\label{bigblocking}
If Lemma \ref{blocksorguts} (ii) holds, then $|E(M\del x)|\leq 10$.
\end{lemma}

\begin{lemma}
\label{9elts}
If Lemma \ref{blocksorguts} (iii) holds, then $|E(M\del x)|\leq 11$.
\end{lemma}

\begin{proof}[Proof of Theorem \ref{mainresultgf4}]
In view of the last three lemmas, it suffices to verify that $\mathcal{P}$ contains each $3$-connected $\{U_{2,5}, U_{3,5}\}$-fragile $\gf(4)$-representable matroid with an $\{X_8,Y_8,Y_8^{*}\}$-minor and no $\{P_8^{-},F_7^{=}, (F_7^{=})^{*}\}$-minor having at most 12 elements. This is done in the Appendix \cite{COZarx}.
\end{proof}

\subsection{Fan extensions}

The following theorem describes the structure of the matroids with no $\{X_8,Y_8,Y_8^{*}\}$-minor. Note that $M_{9,9}$ is the rank-$4$ matroid on $9$ elements in Figure \ref{fig:M71M99}. The matroid $M_{7,1}$ is the $7$-element matroid that is obtained from $Y_8$ by deleting the unique point that is contained in the two $4$-element segments of $Y_8$. We label the points of a triangle of $M_{7,1}$ by $\{1,2,3\}$ as in Figure \ref{fig:M71M99}.

\begin{figure}[tbp]
        \centering
        \includegraphics{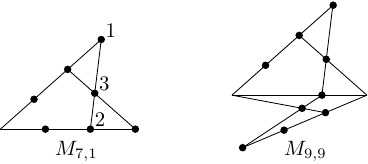}
        
        \caption{The matroids $M_{7,1}$ and $M_{9,9}$.}\label{fig:M71M99}
\end{figure}

\begin{theorem}
 \label{computer2}
  Let $M'$ be a $3$-connected $\{U_{2,5}, U_{3,5}\}$-fragile $\gf(4)$-representable matroid with no $\{P_8^{-},F_7^{=}, (F_7^{=})^{*}\}$-minor. Then $M'$ is isomorphic to a matroid $M$ for which at least one of the following holds: 
    \begin{enumerate}
  \item[(i)] $M$ has an $\{X_8,Y_8,Y_8^{*}\}$-minor;
  \item[(ii)] $M\in \{M_{9,9}, M_{9,9}^{*}\}$;
  \item[(iii)] $M$ or $M^{*}$ can be obtained from $U_{2,5}$ (with ground set $\{a,b,c,d,e\}$) by gluing wheels to $(a,c,b)$,$(a,d,b)$, $(a,e,b)$;
  \item[(iv)] $M$ or $M^{*}$ can be obtained from $U_{2,5}$ (with ground set $\{a,b,c,d,e\}$) by gluing wheels to $(a,b,c)$, $(c,d,e)$;
  \item[(v)] $M$ or $M^{*}$ can be obtained from $M_{7,1}$ by gluing a wheel to $(1,3,2)$.
  \end{enumerate}
\end{theorem}

\begin{proof}
Assume $M$ has no $\{X_8,Y_8,Y_8^{*}\}$-minor. For (ii), we show in Lemma 1 of the Appendix \cite{COZarx} that the matroids $M_{9,9}$ and $M_{9,9}^{*}$ are splitters for the class of $3$-connected $\{U_{2,5}, U_{3,5}\}$-fragile $\gf(4)$-representable matroids with no $\{P_8^{-},F_7^{=}, (F_7^{=})^{*}\}$-minor.

We may therefore assume $M$ has no $\{M_{9,9}, M_{9,9}^{*}, X_8,Y_8,Y_8^{*}\}$-minor. To show that (iii), (iv), or (v) holds, we use the main result of  \cite{chun2013fan} called the ``Fan Lemma'', which reduces the proof to showing that extensions and coextensions of the $9$-element matroids with this structure also have this structure. These verifications are completed in Lemmas 2 through 7 of the Appendix \cite{COZarx}.
\end{proof}

\section{From fragility to relaxations}

We use the following result of Mayhew, Whittle, and Van Zwam \cite[Lemma 8.2]{mayhew2010stability}.

\begin{lemma}
\label{whirls}
Let $M$ be a $3$-connected $U_{2,4}$-fragile matroid that has no $\{U_{2,6}, U_{4,6}\}$-minor. Then exactly one of the following holds.
\begin{enumerate}
\item[(i)] $M$ has rank or corank two;
\item[(ii)] $M$ has an $\{F_7^{-}, (F_7^{-})^{*}\}$-minor; 
\item[(iii)] $M$ has rank and corank at least $3$ and is a whirl.
\end{enumerate}
\end{lemma}

We show next that $P_8^{-},F_7^{=}, (F_7^{=})^{*}$ do not arise from circuit-hyperplane relaxation of a $\gf(4)$-representable matroid.

\begin{lemma}
\label{nobadminor}
Let $M$ and $M'$ be $\gf(4)$-representable matroids such that $M$ is connected, $M'$ is $3$-connected, and $M'$ is obtained from $M$ by relaxing a circuit-hyperplane $X$. Then $M'$ has no $\{P_8^{-},F_7^{=}, (F_7^{=})^{*}\}$-minor.
\end{lemma}

\begin{proof}
Assume that $M'$ has a $\{P_8^{-},F_7^{=}, (F_7^{=})^{*}\}$-minor. Since $M'$ is obtained from $M$ by relaxing $X$, it follows from Theorem \ref{fragilitymain} and Lemma \ref{whirls} that $M'$ is $\{U_{2,5},U_{3,5}\}$-fragile. Each of the matroids in $\{P_8^{-},F_7^{=}, (F_7^{=})^{*}\}$ has a $\{U_{2,5},U_{3,5}\}$-minor, so if $C$ and $D$ are such that $M'/C\del D\cong N'$ for some $N'\in \{P_8^{-},F_7^{=}, (F_7^{=})^{*}\}$, then $C\subseteq X$ and $D\subseteq E(M')-X$ since the elements of $X$ are nondeletable and the elements of $E(M')-X$ are noncontractible by Theorem \ref{fragilitymain}. But then it follows from Lemma \ref{Hminors} that $N'$ can be obtained from $M/C\del D$ by relaxing the circuit-hyperplane $X-C$. It follows that $M/C\del D\cong N$ for some $N\in \{P_8,F_7^{-}, (F_7^{-})^{*}\}$, a contradiction because $M$ is $\gf(4)$-representable. 
\end{proof}

We can now describe the structure of the $\gf(4)$-representable matroids that are circuit-hyperplane relaxations of $\gf(4)$-representable matroids.

\begin{theorem}
\label{main}
Let $M$ and $M'$ be $\gf(4)$-representable matroids such that $M$ is connected, $M'$ is $3$-connected, and $M'$ is obtained from $M$ by relaxing a circuit-hyperplane. Then at least one of the following holds.
  \begin{enumerate}
  \item[(a)] $M'$ is a whirl;
  \item[(b)] $M'\in \{M_{9,9}, M_{9,9}^{*}\}$; 
 \item[(c)] $M'$ or $(M')^{*}$ can be obtained from $U_{2,5}$ (with groundset $\{a,b,c,d,e\}$) by gluing wheels to $(a,c,b)$,$(a,d,b)$; 
  \item[(d)] $M'$ or $(M')^{*}$ can be obtained from $U_{2,5}$ (with groundset $\{a,b,c,d,e\}$) by gluing wheels to $(a,b,c)$,$(c,d,e)$; 
  \item[(e)] $M'$ or $(M')^{*}$ can be obtained from $M_{7,1}$ by gluing a wheel to $(1,3,2)$;   
    \item[(f)] there is some path sequence that describes $M'$.
  \end{enumerate}
\end{theorem}

\begin{proof}
It follows from Theorem \ref{fragilitymain} that $M'$ is either $U_{2,4}$-fragile or $\{U_{2,5}, U_{3,5}\}$-fragile. If $M'$ is $U_{2,4}$-fragile, then it follows from Lemma \ref{whirls} that $M'$ is a whirl. We may therefore assume that $M'$ is $\{U_{2,5}, U_{3,5}\}$-fragile. It follows from Lemma \ref{nobadminor} that $M'$ has no $\{P_8^{-},F_7^{=}, (F_7^{=})^{*}\}$-minor. Then, by Theorem \ref{computer2} and Theorem \ref{fragilitymain}, one of $(b)$ through $(e)$ holds or else $M'$ has an $\{X_8,Y_8,Y_8^{*}\}$-minor. Note that outcome (iii) of Theorem \ref{computer2} corresponds to outcome (c) here, since a matroid or its dual that is obtained from $U_{2,5}$ (with groundset $\{a,b,c,d,e\}$) by gluing wheels onto all three of the triangles $(a,c,b)$,$(a,d,b)$,$(a,e,b)$ does not have a basis of nondeletable elements and a cobasis of noncontractible elements, and therefore cannot be obtained by relaxing a circuit-hyperplane. We can see this by the following counting argument. Observe that the rank of a matroid obtained from $U_{2,5}$ (with groundset $\{a,b,c,d,e\}$) by gluing wheels $A$, $B$ and $C$ onto the triangles $(a,c,b)$,$(a,d,b)$,$(a,e,b)$ is $r(A)+r(B)+r(C)-4$. But the nondeletable elements of this matroid are precisely the rim elements of the wheels of which there are $r(A)+r(B)+r(C)-3$. Hence any cobasis must contain a nondeletable element $e$. Since this matroid has $M_{9,18}$ as a minor (see Appendix \cite[Lemma 2]{COZarx}), $M$ has no essential elements, which implies that $e$ must be contractible.

Finally, if $M'$ has an $\{X_8,Y_8,Y_8^{*}\}$-minor, then it follows from Theorem \ref{mainresultgf4} that (f) holds.
\end{proof}

We can now show that if $M$ and $M'$ are $\gf(4)$-representable matroids such that $M'$ is obtained from $M$ by relaxing a circuit-hyperplane, then $M'$ has path width $3$.

\begin{proof}[Proof of Theorem 1.]
If $M$ is not connected, then it follows from Lemma \ref{disconnected} that $M'$ has path width $3$. We may therefore assume that $M$ is connected. Then, by Lemma \ref{nt2sep}, $M'$ can be obtained from a matroid in Theorem \ref{main} (a) - (f) by performing some, possibly empty, sequence of series or parallel extensions. The result now follows from the fact that all the matroids in Theorem \ref{main} (a) - (f) have path width $3$.
\end{proof}

\section{Forbidden submatrices}

In this section, we will prove our second characterization, Theorem \ref{badsubmatrix}. Let $M$ be a $\gf(4)$-representable matroid with a circuit-hyperplane $X$. Choose $e\in X$ and $f\in E-X$ such that $B=(X-e)\cup f$ is a basis of $M$. Then we can find a reduced $\gf(4)$-representation of $M$ in block form,
\[C=\kbordermatrix{
&(E-X)-f  & e \\
X-e&A & \underline{1}\\
f &\underline{1}^{T} & 0\\
}.
\] 

Here $A$ is an $(X-e)\times ((E-X)-f)$ matrix over $\gf(4)$, and we have scaled so that every non-zero entry in the row labelled by $f$ and the column labelled by $e$ is $1$. We denote by $A_{ij}$ the entry in row $i$ and column $j$ of $A$.

Let $M'$ be the matroid obtained from $M$ by relaxing the circuit-hyperplane $X$. If $M'$ is $\gf(4)$-representable, then we can find a reduced representation of $M'$ in block form,

\[C'=\kbordermatrix{
&(E-X)-f  & e \\
X-e&A' & \underline{1}\\
f &\underline{1}^{T} & \omega\\
}.
\] 

We have scaled the rows and columns of the matrix such that the entry in the row labelled by $f$ and column labelled by $e$ is $\omega\in \gf(4)-\{0,1\}$, and every remaining entry in row $e$ and column $f$ is a $1$.

We omit the straightforward proof of the following lemma.

\begin{lemma}
$A_{ij}=0$ if and only if $A_{ij}'=0$.
\end{lemma}

Next we show that the only non-zero entries of $A'$ are $1$ and $\omega$.

\begin{lemma}
\label{twononzero}
$A_{ij}'\neq \omega+1$.
\end{lemma}

\begin{proof}
Suppose $A_{ij}'=\omega+1$. Then $C'$ has a submatrix
\[C'[\{i,f\}, \{e,j\}]=\kbordermatrix{
& j  & e \\
i &\omega+1&1\\
f &1&\omega\\
},
\] which has determinant zero. Therefore $B\triangle \{e,f,i,j\}$ is not a basis of the matroid $M[I | C']$. But the corresponding submatrix of $C$ is

\[C[\{i,f\}, \{e,j\}]=\kbordermatrix{
& j  & e \\
i &x&1\\
f &1&0\\
},
\] for some non-zero $x$. Since $C[\{i,f\}, \{e,j\}]$ has non-zero determinant, $B\triangle \{e,f,i,j\}$ is a basis of $M$, and hence of $M'$. Therefore $M'\neq M[I | C']$.
\end{proof}

\begin{lemma}
\label{samerow}
$A_{ij}=A_{ik}$ if and only if $A_{ij}'=A_{ik}'$. Similarly, $A_{ij}=A_{kj}$ if and only if $A_{ij}'=A_{kj}'$
\end{lemma}

\begin{proof}
We show that $A_{ij}=A_{ik}$ implies that $A_{ij}'=A_{ik}'$. The proof of the converse, and the proof of the second statement proceed by similar easy arguments. Suppose that $A_{ij}=A_{ik}$. Then $C$ has a submatrix \[C[\{i,f\}, \{j,k\}]=\kbordermatrix{
& j  & k \\
i &x&x\\
f &1&1\\
},
\] for some non-zero $x$. Since $C[\{i,f\}, \{j,k\}]$ has zero determinant, $B\triangle \{f,i,j,k\}$ is not a basis of $M$, and hence not a basis of $M'=M[I | C']$. Therefore $\det(C'[\{i,f\}, \{j,k\}])=0$, so it follows that $A_{ij}'=A_{ik}'$. 
\end{proof}

The following lemma on diagonal submatrices will be used frequently.

\begin{lemma}
\label{diagonal}
Let \[ \begin{bmatrix}
x&0\\
0&y
\end{bmatrix} \text{ and }
\begin{bmatrix}
a&0\\
0&b
\end{bmatrix}\] be corresponding submatrices of $A$ and $A'$ respectively, where $x,y,a,b$ are non-zero entries. Then $x=y$ if and only if $a\neq b$.
\end{lemma}

\begin{proof}
Adjoining $e$ and $f$ to the specified $2\times 2$ submatrices, we get the $3\times 3$ submatrices 
\[\begin{bmatrix}
x&0&1\\
0&y&1\\
1&1&0
\end{bmatrix} \text{ and } \begin{bmatrix}
a&0&1\\
0&b&1\\
1&1&\omega
\end{bmatrix}.\] These matrices have determinants $x+y$ and $ab\omega+a+b$. Thus if $x=y$, then $a\neq b$. Conversely, if $a\neq b$, then $\{a,b\}=\{1,\omega\}$ by Lemma \ref{twononzero} so $ab\omega+a+b=\omega^2+\omega+1=0$. Hence $x=y$.
\end{proof}

We can now identify all of the forbidden submatrices. We use Lemma \ref{samerow} to identify the first such matrix in the following lemma.

\begin{lemma}
\label{bad1by3}
Neither $A$ nor $A^{T}$ has a submatrix of the form \[\begin{bmatrix}
x&y&z
\end{bmatrix},\] where $x,y,z$ are distinct non-zero entries.
\end{lemma}

\begin{proof}
By Lemma \ref{samerow}, the corresponding submatrix of $A'$ must have the form \[\begin{bmatrix}
a&b&c
\end{bmatrix},\] where $a,b,c$ are distinct non-zero entries, which is a contradiction to Lemma \ref{twononzero}.
\end{proof}

We now use Lemma \ref{samerow} and Lemma \ref{diagonal} to find several more forbidden submatrices.

\begin{lemma}
\label{bad2by3}
$A$ has no submatrices of the following forms, where $x$, $y$, and $z$ are distinct non-zero entries.

 \[(i) \begin{bmatrix}
x&x&0\\
x&0&x
\end{bmatrix};\ (ii) \begin{bmatrix}
x&x&0\\
x&0&y
\end{bmatrix};\ (iii) \begin{bmatrix}
x&x&0\\
y&0&y
\end{bmatrix};\ (iv) \begin{bmatrix}
x&y&0\\
x&0&y
\end{bmatrix};\]

\[(v) \begin{bmatrix}
x&0&0\\
0&y&z
\end{bmatrix};\ (vi) \begin{bmatrix}
x&0&0\\
0&x&0\\
0&0&x
\end{bmatrix};\ (vii) \begin{bmatrix}
x&0&0\\
0&x&0\\
0&0&y
\end{bmatrix};\ (viii) \begin{bmatrix}
x&0&0\\
0&y&0\\
0&0&z
\end{bmatrix}.\]

\end{lemma}

\begin{proof}
Suppose $A$ has the submatrix (i). By applying Lemma \ref{samerow} to the rows and the first column, we deduce that the corresponding submatrix of $A'$ has the form  \[\begin{bmatrix}
a&a&0\\
a&0&a
\end{bmatrix},\] where $a$ is a non-zero entry, a contradiction of Lemma \ref{diagonal}.

Suppose $A$ has the submatrix (ii). By applying Lemma \ref{samerow} to the rows and the first column, and since $A'$ has at most two distinct non-zero entries by Lemma \ref{twononzero}, we deduce that the corresponding submatrix of $A'$ has the form  \[\begin{bmatrix}
a&a&0\\
a&0&b
\end{bmatrix},\] where $a$ and $b$ are the two non-zero entries of $A'$, a contradiction to Lemma \ref{diagonal}.

The proofs for (iii) and (iv) are similar to that for (ii). We omit the details.

Suppose $A$ has the submatrix (v). Then, by two applications of Lemma \ref{diagonal}, the corresponding submatrix of $A'$ must have the form \[\begin{bmatrix}
a&0&0\\
0&a&a
\end{bmatrix},\] for some non-zero entry $a$. This is a contradiction to Lemma \ref{samerow}.

Suppose $A$ has the submatrix (vi). By Lemma \ref{diagonal}, the corresponding submatrix of $A'$ must be a diagonal matrix with distinct non-zero entries, a contradiction to Lemma \ref{twononzero}.

Suppose $A$ has the submatrix (vii). Applying Lemma \ref{diagonal} to the two submatrices the form  \[\begin{bmatrix}
x&0\\
0&y
\end{bmatrix},\] it follows that the corresponding submatrix of $A'$ is
\[\begin{bmatrix}
a&0&0\\
0&a&0\\
0&0&a
\end{bmatrix},\] for some $a$, which is a contradiction to Lemma \ref{diagonal}. 

Suppose $A$ has the submatrix (viii). Then the corresponding submatrix of $A'$ is
\[\begin{bmatrix}
a&0&0\\
0&a&0\\
0&0&a
\end{bmatrix},\] for some $a$. Adjoining $e$ and $f$, we have a submatrix of $C$, \[\begin{bmatrix}
x&0&0&1\\
0&y&0&1\\
0&0&z&1\\
1&1&1&0
\end{bmatrix},\] which has zero determinant, while the corresponding submatrix of $C'$,\[\begin{bmatrix}
a&0&0&1\\
0&a&0&1\\
0&0&a&1\\
1&1&1&\omega
\end{bmatrix},\] has non-zero determinant, a contradiction.
\end{proof}

\begin{lemma}
    \label{freddy}
$A$ has no submatrices of the following forms, where $x$, $y$, and $z$ are distinct non-zero entries:

  \[(i) \begin{bmatrix}
x&y\\
0&x
\end{bmatrix};\ (ii)
\begin{bmatrix}
x&y\\
y&x
\end{bmatrix};\ (iii) \begin{bmatrix}
x&x\\
y&z
\end{bmatrix};\ (iv) \begin{bmatrix}
x&y\\
z&x
\end{bmatrix};\ (v) \begin{bmatrix}
x & y & 0 \\
x & 0 & z
\end{bmatrix}.\] 
\end{lemma}

\begin{proof}
Suppose $A$ has the submatrix (i). Then, adjoining $e$ and $f$, we see that $C$ has the following submatrix with non-zero determinant. \[\begin{bmatrix}
x&y&1\\
0&x&1\\
1&1&0
\end{bmatrix}.\] But then, by Lemma \ref{samerow}, the corresponding submatrix of $C'$ must have the following form.\[\begin{bmatrix}
a&b&1\\
0&a&1\\
1&1&\omega
\end{bmatrix},\] where $\{a,b\} = \{1,\omega\}$ by Lemma \ref{twononzero}. This gives a contradiction because this submatrix of $C'$ has zero determinant. A similar proof handles (ii).

Suppose $A$ has the submatrix (iii). Then, by Lemma \ref{samerow}, in the corresponding submatrix of $A'$, the entries in the first row are the same and the entries in the second row are different. But, by Lemma \ref{twononzero}, there are only two distinct non-zero entries in $A'$, so the entries are the same in one of the columns of $A'$, which is a contradiction to Lemma \ref{samerow}.

Suppose  $A$ has the submatrix (iv). Note that this submatrix has zero determinant. By Lemma \ref{samerow}, the corresponding submatrix of $A'$ must have the following form.  \[\begin{bmatrix}
a&b\\
b&a
\end{bmatrix},\] where $\{a,b\} = \{1,\omega\}$ by Lemma \ref{twononzero}. But this submatrix of $A'$ has non-zero determinant, a contradiction.

Suppose $A$ has the submatrix (v). Then $C$ contains the following submatrix, which does not use its last column:
\[\begin{bmatrix}
x & y & 0\\
x & 0 & z\\
1 & 1 & 1
\end{bmatrix}.\]
This matrix has determinant 0. By Lemmas \ref{twononzero}, \ref{samerow}, and \ref{diagonal}, the corresponding submatrix of $C'$ is
\[\begin{bmatrix}
a & b & 0\\
a & 0 & b\\
1 & 1 & 1
\end{bmatrix},\]
where $\{a,b\} = \{1,\omega\}$. This matrix has non-zero determinant, a contradiction.
\end{proof}

Finally, we find two more $3\times 3$ forbidden submatrices of $A$.

\begin{lemma}
\label{bad3by3}
$A$ has no submatrices of the following forms, where $x$, $y$, and $z$ are distinct non-zero entries:

\[(i) \begin{bmatrix}
x&y&x\\
y&y&0\\
x&0&0
\end{bmatrix};\ (ii) \begin{bmatrix}
x&y&x\\
y&y&0\\
x&0&z
\end{bmatrix}.\] 
\end{lemma}

\begin{proof}
Suppose that $A$ has the submatrix (i). Then, adjoining $e$ and $f$, we see that $C$ has the submatrix \[\begin{bmatrix}
x&y&x&1\\
y&y&0&1\\
x&0&0&1\\
1&1&1&0
\end{bmatrix},\] which has zero determinant. The corresponding submatrix of $C'$ is \[\begin{bmatrix}
a&b&a&1\\
b&b&0&1\\
a&0&0&1\\
1&1&1&\omega
\end{bmatrix},\] for distinct $a,b\in \{1,\omega\}$. This submatrix of $C$ has non-zero determinant, a contradiction.

Suppose that $A$ has the submatrix (ii). Note that the determinant of this submatrix is not zero. By Lemma \ref{twononzero} and Lemma \ref{samerow}, the corresponding submatrix of $A'$ is \[\begin{bmatrix}
a&b&a\\
b&b&0\\
a&0&b
\end{bmatrix},\] for distinct $a,b\in \{1,\omega\}$. This submatrix of $A'$ has zero determinant, which is a contradiction. 
\end{proof}

To prove the main theorem of this section, we need the following theorem \cite[Theorem 5.1]{geelen2000excluded}.

\begin{theorem}
\label{bound}
Minor-minimal non-$\gf(4)$-representable matroids have rank and corank at most 4.
\end{theorem}

We can now prove the main theorem, which we repeat for convenience.

\begin{theorem}
\label{badsubmatrixproof}
There is some matrix $C'$ representing $M'$ if and only if, up to permuting rows and columns, $A$ and $A^{T}$ have no submatrix in the following set, where $x,y,z$ are distinct non-zero elements of $\gf(4)$:
\[\begin{bmatrix}
x&y&z
\end{bmatrix},
\begin{bmatrix}
x&y\\
0&x
\end{bmatrix}, 
\begin{bmatrix}
x&y\\
y&x
\end{bmatrix},
\begin{bmatrix}
x&x\\
y&z
\end{bmatrix},
\begin{bmatrix}
x&y\\
z&x
\end{bmatrix},
\begin{bmatrix}
x&x&0\\
x&0&x
\end{bmatrix},
\begin{bmatrix}
x&x&0\\
x&0&y
\end{bmatrix},
\]

\[\begin{bmatrix}
x&x&0\\
y&0&y
\end{bmatrix},
\begin{bmatrix}
x&y&0\\
x&0&y
\end{bmatrix},
\begin{bmatrix}
x&0&0\\
0&y&z
\end{bmatrix},
\begin{bmatrix}
x&y&0\\
x&0&z	
\end{bmatrix},
\begin{bmatrix}
x&0&0\\
0&x&0\\
0&0&x
\end{bmatrix},
\begin{bmatrix}
x&0&0\\
0&x&0\\
0&0&y
\end{bmatrix},\]

\[
\begin{bmatrix}
x&0&0\\
0&y&0\\
0&0&z
\end{bmatrix},
\begin{bmatrix}
x&y&x\\
y&y&0\\
x&0&0
\end{bmatrix},
\begin{bmatrix}
x&y&x\\
y&y&0\\
x&0&z
\end{bmatrix}.\]
\end{theorem}

\begin{proof}
It follows from Lemmas \ref{bad1by3}, \ref{bad2by3}, \ref{freddy}, and \ref{bad3by3} that both $A$ and $A^T$ have no submatrix on the above list. 

Conversely, suppose that the $\gf(4)$-representable matroid $M$ is chosen to be minimal subject to the property that the relaxation $M'$ is not $\gf(4)$-representable. Then $M'$ has a minor $N$ isomorphic to one of the excluded minors for the class of $\gf(4)$-representable matroids. Assume that $N=M'/C\del D$ for some subsets $C$ and $D$. If there is an element $g$ in both $D$ and the circuit-hyperplane $X$ of $M$, then $M\del g=M'\del g$ by Lemma \ref{Hminors}, so $M$ also has an $N$-minor, contradicting the fact that $M$ is $\gf(4)$-representable. We deduce that $D\subseteq E(M)-X$, and dually, $C\subseteq X$. Now if $|D|\geq 2$, then there is some element $g$ in both $D$ and $E(M')-(X\cup f)$, so relaxing the circuit-hyperplane $X$ of $M\del g$ gives $M'\del g$ that is not $\gf(4)$-representable, which contradicts the minimality of $M$. Therefore $|D|\leq 1$, and by a dual argument, there is no element $g$ in both $C$ and $X-e$, so $|C|\leq 1$. Since we know, by Theorem \ref{bound}, that $|E(N)|\leq 8$, it now follows that $|E(M')|\leq 10$. The computations in the Appendix \cite{COZarx} show that $M'$ must have a submatrix from the above list. 
\end{proof}

\subsection*{Acknowledgements}
The authors thank the anonymous referee for some helpful suggestions that improved the exposition.

\appendix



\section*{Supplementary material}

```python
# Excluded minors for GF(4)-representable matroids:
ExMinGF4 = [matroids.Uniform(2,6), matroids.Uniform(4,6),
            matroids.named_matroids.P6(),
            matroids.named_matroids.NonFano(),
            matroids.named_matroids.NonFano().dual(),
            matroids.named_matroids.P8(),
            matroids.named_matroids.P8pp()]

def is_relaxation_quaternary(A):
    """
    Tests if the circuit-hyperplane relaxation of the matroid
    associated with A as described above (so M = M[I C]) has
    an excluded minor isomorphic to one of the GF(4)-excluded
    minors.
    """
    T = Matrix(GF4, A.nrows(), 1)
    for r in range(A.nrows()):
        T[r,0] = 1
    T2 = A.augment(T)
    T = Matrix(GF4, 1, T2.ncols())
    for c in range(T2.ncols()-1):
        T[0,c] = 1
    C = T2.stack(T)
    M = Matroid(reduced_matrix=C)
    X = M.groundset_list()[0:(C.nrows()-1)] + [M.groundset_list()[-1]]
    B = list(M.bases())
    B.append(frozenset(X))
    M2 = Matroid(bases=B)
    return not any(M2.has_minor(N) for N in ExMinGF4)
```

We start by treating the $1 \times n$ matrices $A$ separately. We generate all of them, and save them into a list `A_rows`. Then we test if their relaxation is quaternary by an excluded-minor check, and save the indices of the ones that survive in a list `legits[(1,n)]`. The others go into a list called `rejects[(1,n)]`. Next, we fill a list `legits` with $r \times n$ matrices, encoded as a list of indices into `A_rows`. For each new row, we use the matrices in `legits` one level below to filter which ones have no proper submatrices that were excluded. The new ones are then tested for whether their relaxation is quaternary, and moved into `legits` or `rejects` accordingly. To speed things up, we assume the rows of our matrices are sorted in ascending order with respect to the aforementioned indices.

```
In [3]:  A_rows = {0: []}
         legits = {(0,0): [tuple([])], (0,1): [tuple([])]}
         rejects = {}
         GF4elts = [GF4(0), GF4(1), a, b]

         def make_valid_rows(n):
             A_rows[n] = []
             legits[(1,n)] = []
             rejects[(1,n)] = []

             for i in range(4^n):
                 k = i
                 v = [GF4elts[0]]*n   # row vector to be filled
                 for r in range(n):
                     v[r] = GF4elts[mod(k,4)]
                     k = k // 4
                 A_rows[n].append(tuple(v))
                 if is_relaxation_quaternary(Matrix(GF4, 1, n , v)):
                     legits[(1,n)].append(tuple([len(A_rows[n])-1]))
                 else:
                     rejects[(1,n)].append(tuple([len(A_rows[n])-1]))


In [4]:  def rowlist_to_matrix(L, n):
             A = Matrix(GF4, len(L), n)
             for r in range(len(L)):
                 for c in range(n):
                     A[r,c] = A_rows[n][L[r]][c]
             return A

         def matrix_to_rowlist(A):
             L = []
             n = A.ncols()
             for r in A.rows():
                 v = tuple(r)
                 L.append(A_rows[n].index(v))
             return tuple(sorted(L))
```

```
In [5]: def contains_reject(A):
            r = A.nrows()
            c = A.ncols()
            for ci in range(1,c):
                for S in Subsets(range(c),ci):
                    if r <= ci:
                        Ap = A[:,list(S)]
                        rj = r
                        cj = ci
                    else:
                        Ap = A[:,list(S)].transpose()
                        rj = ci
                        cj = r
                    if matrix_to_rowlist(Ap) in rejects[(rj,cj)]:
                        return True
            return False

        def add_row(r,n):
            """
            Assuming A_rows[n], legits[(r-1,n)], and rejects[(r-1,n)] are
            complete, generate legits[(r,n)] and rejects[(r,n)].
            """
            legits[(r,n)] = []
            rejects[(r,n)] = []
            for A1 in legits[(r-1,n)]:
                for v in legits[(1,n)]:
                    if v[0] >= A1[-1]:  # row vectors are ordered.
                        # check if all submatrices, obtained by removing a
                        # row, occur as valid.
                        contains_smaller_reject = False
                        for i in range(r-1):
                            if (A1[0:i]+A1[i+1:r-1]+v) not in legits[(r-1,n)]:
                                contains_smaller_reject = True
                                break
                        if not contains_smaller_reject and not contains_reject
(rowlist_to_matrix(A1 + v, n)):
                            if is_relaxation_quaternary(rowlist_to_matrix(A1 +
v, n)):
                                legits[(r,n)].append(A1+v)
                            else:
                                rejects[(r,n)].append(A1+v)
```

```
In [6]: make_valid_rows(1)
        print legits
        print rejects

        {(0, 1): [()], (0, 0): [()], (1, 1): [(0,), (1,), (2,), (3,)]}
        {(1, 1): []}
```

```
In [7]: make_valid_rows(2)
        add_row(2,2)
        print len(rejects[(2,2)])

        33
```

```
In [8]: make_valid_rows(3)
        add_row(2,3)
        print len(rejects[(2,3)])
        add_row(3,3)
        print len(rejects[(3,3)])

        117
        99
```

```
In [9]: make_valid_rows(4)
        add_row(2,4)
        print len(rejects[(2,4)])
        add_row(3,4)
        print len(rejects[(3,4)])
        add_row(4,4)
        print len(rejects[(4,4)])

        0
        0
        0
```

```
In [10]: save(A_rows, "A_rows.sobj")
         save(legits, "legits.sobj")
         save(rejects, "rejects.sobj")
```

```
In [11]: A_rows = load("A_rows.sobj")
         legits = load("legits.sobj")
         rejects = load("rejects.sobj")
```

We condense the information gained as follows. Each forbidden submatrix gives rise to several entries in our `rejects` list through column permutations, scaling the entire matrix by a nonzero constant, and through applying a field automorphism followed by zero or more of the above. The encoding of the matrices as a list of indices into `A_rows` gives a natural ordering. We determine and print the lexicographically least of each equivalence class.

```
In [12]:  def Abar(A):
              """
              Applies the nontrivial GF(4)-automorphism to each entry of
              matrix A.
              """
              B = copy(A)
              for i in range(B.nrows()):
                  for j in range(B.ncols()):
                      B[i,j] = B[i,j]^2
              return B

          def canonical_rep_matrix(L, n):
              """
              INPUT:

              - L -- a matrix A, represented as a tuple of row indices
              - n -- the number of columns in A

              OUTPUT:

              - B -- a matrix, represented as a tuple of row indices, equivalent
                to A under the operations described above, and with minimum
                vector of indices.
              """
              Am = rowlist_to_matrix(L,n)
              Bm = Abar(Am)
              starters = [Am, a*Am, b*Am, Bm, a*Bm, b*Bm]
              candidates = []
              for T in starters:
                  for P in Permutations(range(n)):
                      candidates.append(matrix_to_rowlist(T[:,list(P)]))
              return min(candidates)

In [13]:  A = Matrix(GF4, 2,2, [0,a,1,a+1])
          print A
          L = matrix_to_rowlist(A)
          print L

          [    0     w]
          [    1 w + 1]
          (8, 13)

In [14]:  K = canonical_rep_matrix(L,2)
          print K
          print rowlist_to_matrix(K,2)

          (1, 11)
          [    1     0]
          [w + 1     w]
```

```
In [15]:  rejects_minimal = {}
          rejects_dropped = {}
          for T in [(2,2),(1,3),(2,3),(3,3)]:
              rejects_minimal[T] = []
              rejects_dropped[T] = []
              for L in rejects[T]:
                  K = canonical_rep_matrix(L, T[1])
                  if L == K:
                      rejects_minimal[T].append(L)
                  else:
                      rejects_dropped[T].append(K)
```

```
In [16]:  for T in [(2,2), (1,3), (2,3), (3,3)]:
              print "Statistics for size ", T
              print len(rejects[T]), " should be equal to ", len(rejects_minimal
          [T]), "+", len(rejects_dropped[T])
              for L in rejects_dropped[T]:
                  if not L in rejects_minimal[T]:
                      # This should never happen
                      print "The following canonical representative is absent fr
          om our list: ", L
              for L in rejects_minimal[T]:
                  print rowlist_to_matrix(L, T[1])
                  print " "
```

```
Statistics for size  (2, 2)
33  should be equal to  5 + 28
[1 0]
[w 1]

[    1    1]
[w + 1    w]

[    w    1]
[w + 1    1]

[w 1]
[1 w]

[    w    1]
[w + 1    w]

Statistics for size  (1, 3)
6  should be equal to  1 + 5
[w + 1    w    1]

Statistics for size  (2, 3)
117  should be equal to  6 + 111
[    1    0    0]
[    0 w + 1    w]

[1 1 0]
[1 0 1]
```

```
[1 1 0]
[1 0 w]

[1 1 0]
[w 0 w]

[w 1 0]
[w 0 1]

[    w    1      0]
[    w    0  w + 1]

Statistics for size  (3, 3)
99  should be equal to  5 + 94
[1 0 0]
[0 1 0]
[0 0 1]

[1 0 0]
[0 1 0]
[0 0 w]

[    1    0      0]
[    0    w      0]
[    0    0  w + 1]

[1 0 0]
[w w 0]
[1 w 1]

[    1    1      0]
[w + 1    0      w]
[w + 1    1  w + 1]
```

This concludes the proof of Theorem 2.

# Part II. Structure.

We study the quaternary, $\{U_{2,5}, U_{3,5}\}$-fragile matroids with no minor isomorphic to a matroid in $\{P_8^-, F_7^=, (F_7^=)^*\}$.

```
In [17]:  A = Matrix(GF4, [[1,1,1],[1,a,b]])
          U25 = Matroid(reduced_matrix=A)
          U25
```

```
Out[17]:  Quaternary matroid of rank 2 on 5 elements
```

```
In [18]:  def has_U25_or_U35(M):
              d = 1
              while d > 0:
                  N = M.simplify().cosimplify()
                  d = len(M) - len(N)
                  M = N
              if M.full_corank() == 2:
                  return len(M) >= 5
              return M.has_line_minor(5)

          def is_fragile(M):
              # NOTE: does not check whether M has a U25- or U35-minor
              for e in M.groundset():
                  if has_U25_or_U35(M \ e) and has_U25_or_U35(M / e):
                      return False
              return True
```

```
In [19]:  # P8m is a circuit-hyperplane relaxation of P8, obtained by relaxing
          # the top or bottom plane of the twisted cube.
          P8m1 = matroids.named_matroids.P8()
          P8m = Matroid(groundset='abcdefgh', field=GF4,
                        reduced_matrix=[[b,1,1,0],
                                        [1,1,0,1],
                                        [1,0,b,1],
                                        [0,1,a,a]])
          setprint([B for B in P8m.bases() if not B in P8m1.bases()])
          setprint([B for B in P8m1.bases() if not B in P8m.bases()])
```

```
[{'b', 'c', 'f', 'g'}]
[]
```

```
In [20]:    # F7mm is obtained from F7 by relaxing two lines.
            F7mm = Matroid(groundset='abcdefg', field=GF4,
                        reduced_matrix=[[1,1,0,1],
                                        [1,0,1,1],
                                        [0,a,1,1]])
            F7mm.show()
            F7mmd = F7mm.dual()
```

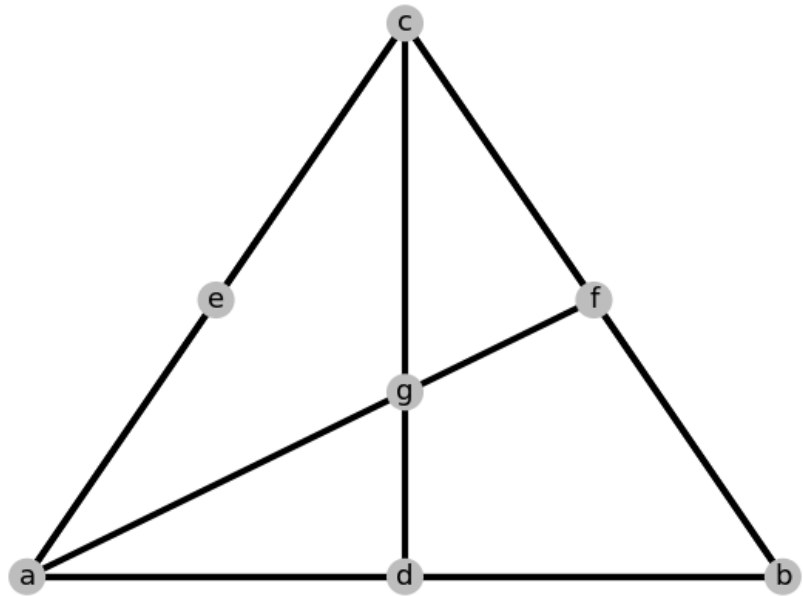

## Generation of the 3-connected matroids in the class with at most 9 elements

We will generate the $\{U_{2,5}, U_{3,5}\}$-fragile, $\mathrm{GF}(4)$-representable, 3-connected matroids with at most 9 elements. By the Splitter Theorem, then, we can generate all these matroids by repeatedly extending or coextending by a single element. The following method does this.

We assume that the elements are labeled $\{0, 1, \ldots, n-1\}$, and will use $n$ for the label of the new element.

```
In [21]: def plus1(X):
             res = []
             for M in X:
                 res.extend(M.linear_extensions(element=len(M), simple=True))
                 res.extend(M.linear_coextensions(element=len(M),
                            cosimple=True))
             res = [M for M in res if not any(M.has_minor(N) for N in [P8m, F7m
         m, F7mmd]) and is_fragile(M)]
             res = get_nonisomorphic_matroids(res)
             return res
```

```
In [22]: %%time
         cat = {}
         cat[5] = [U25, U25.dual()]
         print "Number of 5-element matroids: ", len(cat[5])
         cat[6] = plus1(cat[5])
         print "Number of 6-element matroids: ", len(cat[6])
         cat[7] = plus1(cat[6])
         print "Number of 7-element matroids: ", len(cat[7])
         cat[8] = plus1(cat[7])
         print "Number of 8-element matroids: ", len(cat[8])
         cat[9] = plus1(cat[8])
         print "Number of 9-element matroids: ", len(cat[9])
```

```
Number of 5-element matroids:  2
Number of 6-element matroids:  1
Number of 7-element matroids:  4
Number of 8-element matroids:  8
Number of 9-element matroids:  20
CPU times: user 5.4 s, sys: 135 ms, total: 5.54 s
Wall time: 5.44 s
```

Note that different runs of this generation code are not guaranteed to produce the matroids in the same order (though the sets will always be the same). In order to make the code below reproducible, we sort the lists. Through experimentation, we learned that the tuple $(r, b, c, d)$, where $r$ is the rank, $b$ is the number of bases, $c$ is the number of circuits, and $d$ is the number of cocircuits, uniquely identifies each matroid in our class up to nine elements.

```
In [23]: def f(M):
             return (M.rank(), M.bases_count(),
                     len(M.circuits()), len(M.cocircuits()))

         for i in cat:
             cat[i] = sorted(cat[i], key=f)
```

We go through the matroids generated, printing the circuit closures for easy cross referencing to the figures from [1].

```
In [24]: cat[5][0].representation(reduced=True,B=[0,1],labels=False)

Out[24]: [   1    1    1]
         [   1    w w + 1]


In [25]: setprint(cat[6][0].circuit_closures())
         cat[6][0].is_isomorphic(matroids.named_matroids.Q6())

         {2: {{0, 1, 4}, {2, 3, 4}}, 3: {{0, 1, 2, 3, 4, 5}}}

Out[25]: True


In [26]: # M7,1
         setprint(cat[7][1].circuit_closures())
         print cat[7][3].dual().is_isomorphic(cat[7][1])

         {2: {{0, 1, 4}, {0, 3, 6}, {2, 3, 4}, {2, 5, 6}}, 3: {{0, 1, 2, 3, 4
         , 5, 6}}}
         True


In [27]: # M7,0
         setprint(cat[7][0].circuit_closures())
         print cat[7][2].dual().is_isomorphic(cat[7][0])

         {2: {{0, 1, 4, 6}, {3, 5, 6}, {2, 3, 4}}, 3: {{0, 1, 2, 3, 4, 5, 6}}
         }
         True


In [28]: # Y8
         setprint(cat[8][0].circuit_closures())
         print cat[8][7].dual().is_isomorphic(cat[8][0])

         {2: {{1, 2, 7}, {0, 1, 4, 6}, {2, 3, 4}, {3, 5, 6, 7}}, 3: {{0, 1, 2
         , 3, 4, 5, 6, 7}}}
         True


In [29]: setprint(cat[8][1].circuit_closures())
         print cat[8][1].dual().is_isomorphic(cat[8][1])

         {2: {{0, 1, 4}, {3, 5, 6}, {2, 3, 4}}, 3: {{2, 3, 4, 5, 6}, {0, 1, 2
         , 3, 4}, {0, 1, 4, 6, 7}}, 4: {{0, 1, 2, 3, 4, 5, 6, 7}}}
         True


In [30]: # X8
         setprint(cat[8][2].circuit_closures())
         print cat[8][2].dual().is_isomorphic(cat[8][2])

         {2: {{0, 1, 4, 6}}, 3: {{0, 2, 3, 5}, {2, 3, 4, 7}, {3, 5, 6, 7}}, 4
         : {{0, 1, 2, 3, 4, 5, 6, 7}}}
         True
```

```
In [31]:  setprint(cat[8][3].circuit_closures())
          print cat[8][3].dual().is_isomorphic(cat[8][4])
          setprint(cat[8][4].circuit_closures())
```

```
{2: {{0, 4, 6}, {3, 5, 6}}, 3: {{1, 2, 3, 5, 6}, {0, 3, 4, 5, 6}, {0
, 1, 4, 6, 7}, {2, 3, 4, 7}}, 4: {{0, 1, 2, 3, 4, 5, 6, 7}}}
True
{2: {{0, 1, 4}, {2, 3, 4}, {2, 5, 6}}, 3: {{2, 3, 4, 5, 6}, {0, 1, 2
, 3, 4}, {0, 3, 6, 7}}, 4: {{0, 1, 2, 3, 4, 5, 6, 7}}}
```

```
In [32]:  setprint(cat[8][5].circuit_closures())
          print cat[8][5].dual().is_isomorphic(cat[8][5])
```

```
{2: {{0, 1, 4}, {3, 5, 6}}, 3: {{1, 2, 3, 5, 6}, {0, 1, 4, 6, 7}, {2
, 3, 4, 7}}, 4: {{0, 1, 2, 3, 4, 5, 6, 7}}}
True
```

```
In [33]:  setprint(cat[8][6].circuit_closures())
          print cat[8][6].dual().is_isomorphic(cat[8][6])
```

```
{2: {{0, 1, 4}, {2, 3, 4}}, 3: {{1, 2, 5, 6}, {0, 1, 2, 3, 4}, {0, 1
, 4, 6, 7}, {3, 5, 6, 7}}, 4: {{0, 1, 2, 3, 4, 5, 6, 7}}}
True
```

```
In [34]:  X8 = cat[8][2]
          Y8 = cat[8][0]
          Y8d = cat[8][7]
```

```
In [35]:  print cat[9][0].is_isomorphic(cat[9][10].dual())
          print cat[9][1].is_isomorphic(cat[9][11].dual())
          print cat[9][2].is_isomorphic(cat[9][12].dual())
          print cat[9][3].is_isomorphic(cat[9][13].dual())
          print cat[9][4].is_isomorphic(cat[9][14].dual())
          print cat[9][5].is_isomorphic(cat[9][15].dual())
          print cat[9][6].is_isomorphic(cat[9][16].dual())
          print cat[9][7].is_isomorphic(cat[9][17].dual())
          print cat[9][8].is_isomorphic(cat[9][18].dual())
          print cat[9][9].is_isomorphic(cat[9][19].dual())
```

```
True
True
True
True
True
True
True
True
True
True
```

```
In [36]:  for i in range(10):
              print "M_9,", i
              setprint(cat[9][i].circuit_closures())
```

```
M_9, 0
{2: {{0, 1, 4, 8}, {3, 5, 6}, {2, 3, 4}, {6, 7, 8}}, 3: {{2, 3, 4, 5
, 6}, {3, 5, 6, 7, 8}, {0, 1, 2, 3, 4, 8}, {0, 1, 4, 6, 7, 8}}, 4: {
{0, 1, 2, 3, 4, 5, 6, 7, 8}}}
M_9, 1
{2: {{2, 3, 8}, {0, 4, 6}, {1, 6, 8}, {3, 5, 6}, {4, 7, 8}}, 3: {{2,
3, 4, 7, 8}, {1, 2, 3, 5, 6, 8}, {0, 3, 4, 5, 6}, {0, 1, 4, 6, 7, 8}
}, 4: {{0, 1, 2, 3, 4, 5, 6, 7, 8}}}
M_9, 2
{2: {{0, 1, 4}, {1, 6, 8}, {3, 5, 6}, {2, 3, 8}, {4, 7, 8}}, 3: {{2,
3, 4, 7, 8}, {1, 2, 3, 5, 6, 8}, {0, 1, 4, 6, 7, 8}}, 4: {{0, 1, 2,
3, 4, 5, 6, 7, 8}}}
M_9, 3
{2: {{1, 2, 7}, {0, 1, 4, 6}, {3, 5, 6}}, 3: {{0, 1, 2, 4, 6, 7}, {3
, 5, 6, 7, 8}, {0, 1, 3, 4, 5, 6}, {2, 3, 4, 8}}, 4: {{0, 1, 2, 3, 4
, 5, 6, 7, 8}}}
M_9, 4
{2: {{0, 1, 4, 6}, {2, 3, 4}, {3, 5, 7}}, 3: {{1, 2, 7, 8}, {2, 3, 4
, 5, 7}, {0, 1, 2, 3, 4, 6}, {3, 5, 6, 7, 8}}, 4: {{0, 1, 2, 3, 4, 5
, 6, 7, 8}}}
M_9, 5
{2: {{0, 1, 4}, {1, 7, 8}, {3, 5, 6}, {2, 3, 4}, {0, 6, 8}}, 3: {{2,
3, 4, 5, 6}, {0, 1, 2, 3, 4}, {0, 1, 4, 6, 7, 8}, {0, 3, 5, 6, 8}},
4: {{0, 1, 2, 3, 4, 5, 6, 7, 8}}}
M_9, 6
{2: {{1, 2, 7}, {2, 3, 4}, {3, 5, 7}, {0, 1, 6}}, 3: {{0, 1, 2, 6, 7
}, {1, 2, 3, 4, 5, 7}, {0, 1, 4, 6, 8}, {3, 5, 6, 7, 8}}, 4: {{0, 1,
2, 3, 4, 5, 6, 7, 8}}}
M_9, 7
{2: {{0, 1, 4}, {3, 5, 8}, {2, 3, 4}, {6, 7, 8}}, 3: {{1, 2, 5, 6},
{3, 5, 6, 7, 8}, {0, 1, 2, 3, 4}, {0, 1, 4, 6, 7, 8}, {2, 3, 4, 5, 8
}}, 4: {{0, 1, 2, 3, 4, 5, 6, 7, 8}}}
M_9, 8
{2: {{0, 1, 4}, {3, 7, 8}, {2, 3, 4}, {0, 6, 8}, {2, 5, 6}}, 3: {{2,
3, 4, 5, 6}, {0, 1, 2, 3, 4}, {0, 2, 5, 6, 8}, {2, 3, 4, 7, 8}, {0,
1, 4, 6, 8}, {0, 3, 6, 7, 8}}, 4: {{0, 1, 2, 3, 4, 5, 6, 7, 8}}}
M_9, 9
{2: {{0, 1, 4}, {5, 6, 8}, {3, 7, 8}, {2, 3, 4}}, 3: {{2, 3, 4, 7, 8
}, {3, 5, 6, 7, 8}, {0, 1, 2, 3, 4}, {1, 2, 5, 6, 8}, {0, 1, 4, 6, 7
}}, 4: {{0, 1, 2, 3, 4, 5, 6, 7, 8}}}
```

Our computation, at first, applies to the matroids with no minor in $\{X_8, Y_8, Y_8^*\}$, as well as no minor in $\{F_7^=, (F_7^=)^*, P_8^-\}$. That rules out the following:

```
In [37]:  [i for i in range(10) if cat[9][i].has_minor(X8) or
                                    cat[9][i].has_minor(Y8) or
                                    cat[9][i].has_minor(Y8d)]
```

Out[37]:  [3, 4, 6]

```
In [38]:  excluded_minors = [P8m, F7mm, F7mmd, X8, Y8, Y8d]
```

From now, let $\mathcal{M}$ be the class of $\{U_{2,5}, U_{3,5}\}$-fragile quaternary matroids with no minor in $\{X_8, Y_8, Y_8^*, F_7^=, (F_7^=)^*, P_8^-\}$.

**Lemma 1.** *The matroid $M_{9,9}$ is a splitter for $\mathcal{M}$.*

```
In [39]:  print [M for M in cat[9][9].linear_extensions(simple=True)
                  if is_fragile(M) and not any(M.has_minor(N) for N in excluded_m
          inors)]
          print [M for M in cat[9][9].linear_coextensions(cosimple=True)
                  if is_fragile(M) and not any(M.has_minor(N) for N in excluded_m
          inors)]

          []
          []
```

```
In [40]:  def print_fans(M):
              print "Triangles: "
              setprint([C for C in M.circuits() if len(C) == 3])
              print "Triads: "
              setprint([C for C in M.cocircuits() if len(C) == 3])
```

Consider the following set of three disjoint triads of $M_{9,18}$: $\mathcal{F} = \{(2,3,5),(7,4,0),(1,6,8)\}$, where the ordering of the elements was chosen so the deletable element is in the center. We will show:

**Lemma 2.** *Every 3-connected matroid in $\mathcal{M}$ with a minor isomorphic to $M_{9,18}$ is a fan-extension of $M_{9,18}$ with respect to $\mathcal{F}$.*

```
In [41]:  M18 = cat[9][18]
          print_fans(M18)
          print "Deletable elements: "
          setprint([e for e in M18.groundset() if (M18 \ e).has_minor(U25)])

          Triangles:
          []
          Triads:
          [{0, 5, 8}, {1, 6, 8}, {2, 3, 5}, {1, 2, 7}, {0, 4, 7}]
          Deletable elements:
          [3, 4, 6]
```

```
In [42]: exList = [M for M in M18.linear_extensions(9, simple=True)
                    if is_fragile(M) and not any(M.has_minor(N) for N in exclude
         d_minors)]
         exList.extend([M for M in M18.linear_coextensions(9, cosimple=True)
                         if is_fragile(M) and not any(M.has_minor(N) for N in ex
         cluded_minors)])
         print len(exList)
```

2

```
In [43]: for M in exList:
             print M.rank()
             print_fans(M)
             print " "
```

5
Triangles:
[{4, 7, 9}, {1, 6, 9}, {2, 3, 9}]
Triads:
[{0, 4, 7}, {2, 3, 5}, {1, 6, 8}, {0, 5, 8}]

5
Triangles:
[{6, 8, 9}, {3, 5, 9}, {0, 4, 9}]
Triads:
[{1, 2, 7}, {0, 4, 7}, {2, 3, 5}, {1, 6, 8}]

```
In [44]: doubleList = []
         for N in exList:
             doubleList.extend([M for M in N.linear_extensions(10,
                                 simple=True)
                                 if is_fragile(M) and not any(M.has_minor(N) for
         N in excluded_minors)])
             doubleList.extend([M for M in N.linear_coextensions(10,
                                 cosimple=True)
                                 if is_fragile(M) and not any(M.has_minor(N) for
         N in excluded_minors)])
         print len(doubleList)
```

8

```
In [45]: for M in doubleList:
             print M.rank()
             print_fans(M)
             print " "
```

5
Triangles:
[{4, 7, 9}, {1, 6, 9}, {2, 3, 9}, {6, 8, 10}, {3, 5, 10}, {0, 4, 10}
]
Triads:
[{1, 6, 8}, {2, 3, 5}, {0, 4, 7}]

6
Triangles:
[{4, 7, 9}]
Triads:
[{7, 9, 10}, {1, 2, 10}, {0, 5, 8}, {1, 6, 8}, {2, 3, 5}, {0, 4, 7}]

6
Triangles:
[{2, 3, 9}]
Triads:
[{2, 9, 10}, {1, 7, 10}, {0, 5, 8}, {1, 6, 8}, {2, 3, 5}, {0, 4, 7}]

6
Triangles:
[{1, 6, 9}]
Triads:
[{1, 9, 10}, {2, 7, 10}, {0, 5, 8}, {1, 6, 8}, {2, 3, 5}, {0, 4, 7}]

5
Triangles:
[{6, 8, 9}, {3, 5, 9}, {0, 4, 9}, {4, 7, 10}, {1, 6, 10}, {2, 3, 10}
]
Triads:
[{1, 6, 8}, {2, 3, 5}, {0, 4, 7}]

6
Triangles:
[{0, 4, 9}]
Triads:
[{0, 9, 10}, {5, 8, 10}, {1, 6, 8}, {2, 3, 5}, {0, 4, 7}, {1, 2, 7}]

6
Triangles:
[{6, 8, 9}]
Triads:
[{8, 9, 10}, {0, 5, 10}, {1, 6, 8}, {2, 3, 5}, {0, 4, 7}, {1, 2, 7}]

6
Triangles:
[{3, 5, 9}]
Triads:
[{5, 9, 10}, {0, 8, 10}, {1, 6, 8}, {2, 3, 5}, {0, 4, 7}, {1, 2, 7}]

All are fan-extensions, which concludes the proof of the lemma.

**Lemma 3.** *Every 3-connected matroid in $\mathcal{M}$ with an $M_{9,7}$-minor but no minor in $\{X_8, Y_8, Y_8^*\}$ is a fan-extension of $M_{9,7}$ with respect to $\mathcal{F} = \{(4, 2, 3, 5, 8)\}$.*

```
In [46]:  M7 = cat[9][7]
          print_fans(M7)

          Triangles:
          [{0, 1, 4}, {2, 3, 4}, {6, 7, 8}, {3, 5, 8}]
          Triads:
          [{2, 3, 5}]
```

```
In [47]:  exList = [M for M in M7.linear_extensions(9, simple=True)
                      if is_fragile(M) and not any(M.has_minor(N) for N in exclude
          d_minors)]
          exList.extend([M for M in M7.linear_coextensions(9, cosimple=True)
                          if is_fragile(M) and not any(M.has_minor(N) for N in e
          xcluded_minors)])
          print len(exList)

          2
```

```
In [48]:  for M in exList:
              print M.rank()
              print_fans(M)
              print " "

          5
          Triangles:
          [{2, 3, 4}, {6, 7, 8}, {3, 5, 8}]
          Triads:
          [{0, 1, 9}, {2, 4, 9}, {2, 3, 5}]

          5
          Triangles:
          [{2, 3, 4}, {0, 1, 4}, {3, 5, 8}]
          Triads:
          [{5, 8, 9}, {6, 7, 9}, {2, 3, 5}]
```

```
doubleList = []
for N in exList:
    doubleList.extend([M for M in N.linear_extensions(10,
                              simple=True)
                       if is_fragile(M) and not any(M.has_minor(N) for
N in excluded_minors)])
    doubleList.extend([M for M in N.linear_coextensions(10,
                              cosimple=True)
                       if is_fragile(M) and not any(M.has_minor(N) for
N in excluded_minors)])
print len(doubleList)
```

4

```
for M in doubleList:
    print M.rank()
    print_fans(M)
    print " "
```

```
5
Triangles:
[{2, 3, 4}, {6, 7, 8}, {3, 5, 8}, {4, 9, 10}, {0, 1, 10}]
Triads:
[{2, 3, 5}, {2, 4, 9}]

6
Triangles:
[{2, 3, 4}, {3, 5, 8}]
Triads:
[{6, 7, 10}, {5, 8, 10}, {2, 3, 5}, {2, 4, 9}, {0, 1, 9}]

5
Triangles:
[{2, 3, 4}, {0, 1, 4}, {3, 5, 8}, {8, 9, 10}, {6, 7, 10}]
Triads:
[{2, 3, 5}, {5, 8, 9}]

6
Triangles:
[{2, 3, 4}, {3, 5, 8}]
Triads:
[{2, 4, 10}, {0, 1, 10}, {2, 3, 5}, {6, 7, 9}, {5, 8, 9}]
```

All are fan-extensions, which concludes the proof of the lemma.

**Lemma 4.** *Every 3-connected matroid in $\mathcal{M}$ with an $M_{9,15}$-minor but no $M_{9,8}$- or $M_{9,18}$-minor is a fan-extension of $M_{9,15}$ with respect to $\mathcal{F} = \{(7, 4, 0, 6, 8), (2, 3, 5)\}$.*

```
In [51]:  M15 = cat[9][15]
          print_fans(M15)
```

```
Triangles:
[{0, 4, 6}]
Triads:
[{1, 5, 8}, {0, 6, 8}, {2, 3, 5}, {1, 2, 7}, {0, 4, 7}]
```

```
In [52]:  M18d = M18.dual()
          exList = [M for M in M15.linear_extensions(9, simple=True)
                       if is_fragile(M) and not M.has_minor(M18)
                                         and not M.has_minor(M18d)
                                         and not any(M.has_minor(N) for N in exclude
          d_minors)]
          exList.extend([M for M in M15.linear_coextensions(9, cosimple=True)
                       if is_fragile(M) and not M.has_minor(M18)
                                         and not M.has_minor(M18d)
                                         and not any(M.has_minor(N) for N in ex
          cluded_minors)])
          print len(exList)
```

```
2
```

```
In [53]:  for M in exList:
              print M.rank()
              print_fans(M)
              print " "
```

```
5
Triangles:
[{0, 4, 6}, {6, 8, 9}, {3, 5, 9}]
Triads:
[{1, 2, 7}, {0, 4, 7}, {2, 3, 5}, {0, 6, 8}]

5
Triangles:
[{0, 4, 6}, {4, 7, 9}, {2, 3, 9}]
Triads:
[{0, 4, 7}, {2, 3, 5}, {0, 6, 8}, {1, 5, 8}]
```

```
In [54]:  doubleList = []
          for N in exList:
              doubleList.extend([M for M in N.linear_extensions(10,
                                              simple=True)
                                      if is_fragile(M) and not M.has_minor(M18)
                                              and not M.has_minor(M18d)
                                              and not any(M.has_minor(N) for
          N in excluded_minors)])
              doubleList.extend([M for M in N.linear_coextensions(10,
                                              cosimple=True)
                                      if is_fragile(M) and not M.has_minor(M18)
                                              and not M.has_minor(M18d)
                                              and not any(M.has_minor(N) for
          N in excluded_minors)])
          print len(doubleList)

          6


In [55]:  for M in doubleList:
              print M.rank()
              print_fans(M)
              print " "
```

```
5
Triangles:
[{0, 4, 6}, {6, 8, 9}, {3, 5, 9}, {1, 9, 10}, {4, 7, 10}, {2, 3, 10}
]
Triads:
[{0, 6, 8}, {2, 3, 5}, {0, 4, 7}]


6
Triangles:
[{0, 4, 6}, {6, 8, 9}]
Triads:
[{8, 9, 10}, {1, 5, 10}, {0, 6, 8}, {2, 3, 5}, {0, 4, 7}, {1, 2, 7}]


6
Triangles:
[{0, 4, 6}, {3, 5, 9}]
Triads:
[{5, 9, 10}, {1, 8, 10}, {0, 6, 8}, {2, 3, 5}, {0, 4, 7}, {1, 2, 7}]


5
Triangles:
[{0, 4, 6}, {4, 7, 9}, {2, 3, 9}, {1, 9, 10}, {6, 8, 10}, {3, 5, 10}
]
Triads:
[{0, 6, 8}, {2, 3, 5}, {0, 4, 7}]


6
Triangles:
[{0, 4, 6}, {4, 7, 9}]
Triads:
[{7, 9, 10}, {1, 2, 10}, {1, 5, 8}, {0, 6, 8}, {2, 3, 5}, {0, 4, 7}]


6
Triangles:
[{0, 4, 6}, {2, 3, 9}]
Triads:
[{2, 9, 10}, {1, 7, 10}, {1, 5, 8}, {0, 6, 8}, {2, 3, 5}, {0, 4, 7}]
```

These are all fan-extensions, which concludes the proof of the lemma.

**Lemma 5.** *Every 3-connected matroid in $M$ with an $M_{9,2}$-minor is a fan-extension of $M_{9,2}$ with respect to $\mathcal{F} = \{(1, 0, 4, 7), (6, 5, 3, 2)\}$.*

Note that element 8 can be added to either fan.

```
In [56]: M2 = cat[9][2]
         print_fans(M2)
```

Triangles:
[{0, 1, 4}, {3, 5, 6}, {4, 7, 8}, {1, 6, 8}, {2, 3, 8}]
Triads:
[{2, 3, 5}, {0, 4, 7}]

```
In [57]: exList = [M for M in M2.linear_extensions(9, simple=True)
                   if is_fragile(M) and not any(M.has_minor(N) for N in exclude
         d_minors)]
         exList.extend([M for M in M2.linear_coextensions(9, cosimple=True)
                        if is_fragile(M) and not any(M.has_minor(N) for N in ex
         cluded_minors)])
         print len(exList)
```

3

```
In [58]: for M in exList:
             print M.rank()
             print_fans(M)
             print " "
```

5
Triangles:
[{0, 1, 4}, {3, 5, 6}, {4, 7, 8}, {2, 3, 8}]
Triads:
[{5, 6, 9}, {0, 1, 9}, {0, 4, 7}, {2, 3, 5}]

5
Triangles:
[{3, 5, 6}, {0, 1, 4}, {4, 7, 8}]
Triads:
[{7, 8, 9}, {0, 4, 7}, {2, 3, 5}]

5
Triangles:
[{0, 1, 4}, {3, 5, 6}, {2, 3, 8}]
Triads:
[{2, 8, 9}, {0, 4, 7}, {2, 3, 5}]

```
doubleList = []
for N in exList:
    doubleList.extend([M for M in N.linear_extensions(10,
                                simple=True)
                         if is_fragile(M) and not any(M.has_minor(N) for
N in excluded_minors)])
    doubleList.extend([M for M in N.linear_coextensions(10,
                                cosimple=True)
                         if is_fragile(M) and not any(M.has_minor(N) for
N in excluded_minors)])
print len(doubleList)
```

8

```
for M in doubleList:
    print M.rank()
    print_fans(M)
    print " "
```

5
Triangles:
[{0, 1, 4}, {3, 5, 6}, {4, 7, 8}, {2, 3, 8}, {1, 9, 10}, {6, 8, 10}]
Triads:
[{2, 3, 5}, {0, 4, 7}, {0, 1, 9}]

5
Triangles:
[{0, 1, 4}, {3, 5, 6}, {4, 7, 8}, {2, 3, 8}, {6, 9, 10}, {1, 8, 10}]
Triads:
[{2, 3, 5}, {0, 4, 7}, {5, 6, 9}]

6
Triangles:
[{3, 5, 6}, {0, 1, 4}, {2, 3, 8}]
Triads:
[{7, 9, 10}, {2, 8, 10}, {2, 3, 5}, {0, 4, 7}, {0, 1, 9}, {5, 6, 9}]

6
Triangles:
[{3, 5, 6}, {0, 1, 4}, {4, 7, 8}]
Triads:
[{2, 9, 10}, {7, 8, 10}, {2, 3, 5}, {0, 4, 7}, {0, 1, 9}, {5, 6, 9}]

5
Triangles:
[{3, 5, 6}, {0, 1, 4}, {4, 7, 8}, {8, 9, 10}, {2, 3, 10}, {1, 6, 10}
]
Triads:
[{0, 4, 7}, {2, 3, 5}, {7, 8, 9}]

6
Triangles:
[{3, 5, 6}, {0, 1, 4}, {4, 7, 8}]
Triads:
[{2, 9, 10}, {0, 1, 10}, {5, 6, 10}, {2, 3, 5}, {0, 4, 7}, {7, 8, 9}
]

5
Triangles:
[{0, 1, 4}, {3, 5, 6}, {2, 3, 8}, {8, 9, 10}, {4, 7, 10}, {1, 6, 10}
]
Triads:
[{0, 4, 7}, {2, 3, 5}, {2, 8, 9}]

6
Triangles:
[{3, 5, 6}, {0, 1, 4}, {2, 3, 8}]
Triads:
[{7, 9, 10}, {0, 1, 10}, {5, 6, 10}, {2, 3, 5}, {0, 4, 7}, {2, 8, 9}
]

These are all fan-extensions of $M_{9,2}$, which completes the proof.

**Lemma 6.** *Every 3-connected matroid in $\mathcal{M}$ with a proper $M_{9,1}$-minor must have an $M_{9,15}$- or $M_{9,18}$-minor.*

```
In [61]: M1 = cat[9][1]
         exList = [M for M in M1.linear_extensions(9, simple=True)
                      if is_fragile(M) and not M.has_minor(M15)
                                       and not M.has_minor(M18)
                                       and not any(M.has_minor(N) for N in exclude
         d_minors)]
         exList.extend([M for M in M1.linear_coextensions(9, cosimple=True)
                           if is_fragile(M) and not M.has_minor(M15)
                                            and not M.has_minor(M18)
                                            and not any(M.has_minor(N) for N in ex
         cluded_minors)])
         print len(exList)
```
```
         0
```

Hence all matroids arising from $M_{9,1}$ have already been considered.

**Lemma 7.** *Every 3-connected matroid in $\mathcal{M}$ with an $M_{9,0}$-minor but no $M_{8,5}$-, $M_{9,15}$-, $M_{9,7}$-minor is a fan-extension of $M_{9,0}$ with respect to $\mathcal{F} = \{(4, 2, 3, 5, 6, 7, 8)\}$.*

We can exclude $M_{8,5}$ because all of its extensions and coextensions have already been dealt with:

```
In [62]: [i for i in range(20) if cat[9][i].has_minor(cat[8][5])]
```
```
Out[62]: [2, 4, 6, 9, 12, 14, 16, 19]
```

Now we verify the conditions of the Fan Lemma:

```
In [63]: M0 = cat[9][0]
         print_fans(M0)
```
```
         Triangles:
         [{0, 1, 4}, {2, 3, 4}, {3, 5, 6}, {6, 7, 8}, {0, 1, 8}, {0, 4, 8}, {
         1, 4, 8}]
         Triads:
         [{2, 3, 5}, {5, 6, 7}]
```

```
In [64]: S = excluded_minors + [cat[8][5], cat[9][5], cat[9][15], cat[9][7], ca
         t[9][17]]
```

```
In [65]: exList = [M for M in M0.linear_extensions(9, simple=True)
                  if is_fragile(M) and not any(M.has_minor(N) for N in S)]
         exList.extend([M for M in M0.linear_coextensions(9, cosimple=True)
                  if is_fragile(M) and not any(M.has_minor(N)
                                                      for N in S)])
         print len(exList)

         2
```

```
In [66]: for M in exList:
             print M.rank()
             print_fans(M)
             print " "

         5
         Triangles:
         [{3, 5, 6}, {2, 3, 4}, {0, 1, 4}, {6, 7, 8}]
         Triads:
         [{7, 8, 9}, {0, 1, 9}, {5, 6, 7}, {2, 3, 5}]

         5
         Triangles:
         [{2, 3, 4}, {3, 5, 6}, {6, 7, 8}, {0, 1, 8}]
         Triads:
         [{0, 1, 9}, {2, 4, 9}, {5, 6, 7}, {2, 3, 5}]
```

```
In [67]: doubleList = []
         for N in exList:
             doubleList.extend([M for M in N.linear_extensions(10,
                        simple=True)
                        if is_fragile(M) and not any(M.has_minor(N)
                                                      for N in S)])
             doubleList.extend([M for M in N.linear_coextensions(10,
                        cosimple=True)
                        if is_fragile(M) and not any(M.has_minor(N)
                                                      for N in S)])
         print len(doubleList)

         4
```

```
In [68]: for M in doubleList:
             print M.rank()
             print_fans(M)
             print " "
```

```
5
Triangles:
[{3, 5, 6}, {2, 3, 4}, {0, 1, 4}, {6, 7, 8}, {8, 9, 10}, {0, 4, 10},
{1, 4, 10}, {0, 1, 10}]
Triads:
[{5, 6, 7}, {2, 3, 5}, {7, 8, 9}]

6
Triangles:
[{3, 5, 6}, {2, 3, 4}, {6, 7, 8}]
Triads:
[{0, 9, 10}, {1, 9, 10}, {2, 4, 10}, {0, 1, 10}, {2, 3, 5}, {5, 6, 7
}, {0, 1, 9}, {7, 8, 9}]

5
Triangles:
[{2, 3, 4}, {3, 5, 6}, {6, 7, 8}, {0, 1, 8}, {4, 9, 10}, {0, 8, 10},
{1, 8, 10}, {0, 1, 10}]
Triads:
[{2, 3, 5}, {5, 6, 7}, {2, 4, 9}]

6
Triangles:
[{3, 5, 6}, {2, 3, 4}, {6, 7, 8}]
Triads:
[{0, 9, 10}, {1, 9, 10}, {0, 1, 10}, {7, 8, 10}, {2, 3, 5}, {5, 6, 7
}, {2, 4, 9}, {0, 1, 9}]
```

All of these are fan-extensions, which completes the proof of the lemma.

## Path sequences

Our final step is the study of the matroids that do have an $X_8, Y_8, Y_8^*$ minor. We show they can be generated by repeated segment-cosegment exchanges on designated elements. By our main theorem, we only have to check this up to 12 elements.

```
In [69]: def DY(M, segment, parallels=None):
             # First, construct the parallel extension of M
             if parallels is not None:
                 for e in parallels:
                     M = M.linear_extension(element=len(M), chain={e: 1})

             # Next, find a suitable representation of the segment:
             BS = M.augment([], segment)
```

```
        NBS = [e for e in segment if not e in BS]
        B = M.augment(BS)
        A, row_labels, col_labels = M.representation(reduced=True, B=B, la
bels=True)

        # construct the Theta matroid
        row_indices = [row_labels.index(e) for e in BS]
        col_indices = [col_labels.index(e) for e in NBS]
        k = len(segment)
        D = Matrix(A.base_ring(), nrows=k, ncols=k)
        D[0,1] = A[row_indices[1],col_indices[0]] / A[row_indices[0],col_i
ndices[0]]
        D[1,0] = -D[0,1]
        for i in range(k-2):
            D[i+2,0] = A[row_indices[0],col_indices[i]]
            D[i+2,1] = A[row_indices[1],col_indices[i]]
            D[0,i+2] = A[row_indices[0],col_indices[i]]
            D[1,i+2] = A[row_indices[1],col_indices[i]]

        # construct the ground set of the generalized parallel connection
        tmpE = M.groundset_list()
        delset = []
        E = copy(row_labels)
        for i in row_indices:
            E[i] = newlabel(tmpE)
            tmpE.append(E[i])
            delset.append(E[i])
        E.extend(NBS)
        F = copy(col_labels)
        for i in col_indices:
            F[i] = newlabel(tmpE)
            tmpE.append(F[i])
            delset.append(F[i])
        F.extend(BS)
        E.extend(F)

        # construct the new representation
        AA = Matrix(A.base_ring(), nrows=A.nrows() + k - 2, ncols=A.ncols(
)+2)
        AA.set_block(0,0,A)
        AA[row_indices[0], A.ncols()+1] = D[0,1]
        AA[row_indices[1], A.ncols()] = D[1,0]
        for i in range(k-2):
            AA[A.nrows()+i, A.ncols()] = D[i+2,0]
            AA[A.nrows()+i, A.ncols()+1] = D[i+2,1]
        return Matroid(groundset=E,reduced_matrix=AA).delete(delset)

def YD(M, cosegment, series=None):
        return DY(M.dual(), segment=cosegment, parallels=series).dual()
```

The segment and cosegment of $X_8$ are the following:

```
In [70]:  print X8.rank([0,1,4,6])
          print X8.corank([2,3,5,7])
          S = [0,1,4,6]
          C = [2,3,5,7]

          2
          2
```

```
In [71]:  print "Deletable: ", [e for e in X8.groundset() if has_U25_or_U35(X8 \
          e)]
          print "Contractible: ", [e for e in X8.groundset() if has_U25_or_U35(X
          8 / e)]

          Deletable:  [0, 3, 4, 6]
          Contractible:  [1, 2, 5, 7]
```

$X_8$ has an automorphism group whose restriction to $\{0, 4, 6\}$ is $S_3$, and whose restriction to $\{2, 5, 7\}$ is $S_3$:

```
In [72]:  p1 = {0:4, 1:1, 3:3, 4:6, 6:0, 2:7, 5:2, 7:5}
          p2 = {0:0, 1:1, 3:3, 6:4, 4:6, 2:5, 5:2, 7:7}
          X8.is_isomorphism(X8,p1) and X8.is_isomorphism(X8,p2)
```

```
Out[72]:  True
```

```
In [73]:  def move(M):
              """
              Return all matroids that can be produced from M by one allowable m
          ove.

              Fan moves are broken up into Delta-Y and Y-Delta moves. We use the
          fact that in
              a Delta-Y move, the original elements from X8 survive as an intern
          al triangle
              or triad of the fan.
              """
              S = [0,1,4,6]
              C = [2,3,5,7]
              safe_elts = [1,3]
              fanbase = [[0,4,1], [0,6,1], [4,6,1], [2,5,3], [2,7,3], [5,7,3]]

              # Find all allowable segments
              segs = []
              if M.rank(S) == 2:
                  segs.append(S)
              if M.rank(C) == 2:
                  segs.append(C)
              # add allowable triangles
              for T in fanbase:
                  if M.rank(T) == 2:
                      segs.append(T)
```

```
# Find all allowable cosegments
cosegs = []
if M.corank(S) == 2:
    cosegs.append(S)
if M.corank(C) == 2:
    cosegs.append(C)
# add allowable triads
for T in fanbase:
    if M.corank(T) == 2:
        cosegs.append(T)
L = []

# Delta-Y moves
for S in segs:
    Sp = [e for e in S if e not in safe_elts]
    for X in Subsets(Sp):
        if len(S) == 3 or (len(S) == 4 and len(X) >= 1):
            L.append(DY(M,S,X))

# Y-Delta moves
for C in cosegs:
    Cp = [e for e in C if e not in safe_elts]
    for X in Subsets(Cp):
        if len(C) == 3 or (len(C) == 4 and len(X) >= 1):
            L.append(YD(M,C,X))

    return L
```

In [74]:
```python
def grow(pathmatroids, newpathmatroids, n=None):
    """
    Take each of the matroids in `newpathmatroids` and add all pathmat
roids that can be obtained through one
    (generalized) Delta-Y move. Throw away results on more than `n` el
ements, if these occur.
    """
    L = []
    for M in newpathmatroids:
        L.extend(move(M))
    npm = []
    for N in L:
        if (n is None or N.size() <= n) and not any(M.is_isomorphic(N)
for M in pathmatroids):
            pathmatroids.append(N)
            npm.append(N)
    return pathmatroids, npm
```

```
In [75]: pathmatroids = [X8]
         newpathmatroids = [X8]

         pathmatroids, newpathmatroids = grow(pathmatroids, newpathmatroids, 12
         )
         print len(pathmatroids)
         print len(newpathmatroids)
         print len([M for M in pathmatroids if M.size() == 8])
         print len([M for M in pathmatroids if M.size() == 9])
         print len([M for M in pathmatroids if M.size() == 10])
         print len([M for M in pathmatroids if M.size() == 11])
         print len([M for M in pathmatroids if M.size() == 12])
```

```
         13
         12
         3
         4
         4
         2
         0
```

```
In [76]: pathmatroids, newpathmatroids = grow(pathmatroids, newpathmatroids, 12
         )
         print len(pathmatroids)
         print len(newpathmatroids)
         print len([M for M in pathmatroids if M.size() == 8])
         print len([M for M in pathmatroids if M.size() == 9])
         print len([M for M in pathmatroids if M.size() == 10])
         print len([M for M in pathmatroids if M.size() == 11])
         print len([M for M in pathmatroids if M.size() == 12])
```

```
         80
         67
         4
         10
         26
         26
         14
```

```
In [77]: pathmatroids, newpathmatroids = grow(pathmatroids, newpathmatroids, 12
         )
         print len(pathmatroids)
         print len(newpathmatroids)
         print len([M for M in pathmatroids if M.size() == 8])
         print len([M for M in pathmatroids if M.size() == 9])
         print len([M for M in pathmatroids if M.size() == 10])
         print len([M for M in pathmatroids if M.size() == 11])
         print len([M for M in pathmatroids if M.size() == 12])
```

```
405
325
4
18
60
130
193
```

```
In [78]: pathmatroids, newpathmatroids = grow(pathmatroids, newpathmatroids, 12
         )
         print len(pathmatroids)
         print len(newpathmatroids)
         print len([M for M in pathmatroids if M.size() == 8])
         print len([M for M in pathmatroids if M.size() == 9])
         print len([M for M in pathmatroids if M.size() == 10])
         print len([M for M in pathmatroids if M.size() == 11])
         print len([M for M in pathmatroids if M.size() == 12])
```

```
958
553
4
20
84
253
597
```

```
In [79]: pathmatroids, newpathmatroids = grow(pathmatroids, newpathmatroids, 12
         )
         print len(pathmatroids)
         print len(newpathmatroids)
         print len([M for M in pathmatroids if M.size() == 8])
         print len([M for M in pathmatroids if M.size() == 9])
         print len([M for M in pathmatroids if M.size() == 10])
         print len([M for M in pathmatroids if M.size() == 11])
         print len([M for M in pathmatroids if M.size() == 12])
```

```
1540
582
4
20
92
348
1076
```

```
In [80]: pathmatroids, newpathmatroids = grow(pathmatroids, newpathmatroids, 12
         )
         print len(pathmatroids)
         print len(newpathmatroids)
         print len([M for M in pathmatroids if M.size() == 8])
         print len([M for M in pathmatroids if M.size() == 9])
         print len([M for M in pathmatroids if M.size() == 10])
         print len([M for M in pathmatroids if M.size() == 11])
         print len([M for M in pathmatroids if M.size() == 12])
```

```
1874
334
4
20
93
376
1381
```

```
In [81]: pathmatroids, newpathmatroids = grow(pathmatroids, newpathmatroids, 12
         )
         print len(pathmatroids)
         print len(newpathmatroids)
         print len([M for M in pathmatroids if M.size() == 8])
         print len([M for M in pathmatroids if M.size() == 9])
         print len([M for M in pathmatroids if M.size() == 10])
         print len([M for M in pathmatroids if M.size() == 11])
         print len([M for M in pathmatroids if M.size() == 12])
```

```
1984
110
4
20
93
380
1487
```

```
In [82]: pathmatroids, newpathmatroids = grow(pathmatroids, newpathmatroids, 12
         )
         print len(pathmatroids)
         print len(newpathmatroids)
         print len([M for M in pathmatroids if M.size() == 8])
         print len([M for M in pathmatroids if M.size() == 9])
         print len([M for M in pathmatroids if M.size() == 10])
         print len([M for M in pathmatroids if M.size() == 11])
         print len([M for M in pathmatroids if M.size() == 12])
```

```
1999
15
4
20
93
380
1502
```

```
In [83]: pathmatroids, newpathmatroids = grow(pathmatroids, newpathmatroids, 12
         )
         print len(pathmatroids)
         print len(newpathmatroids)
         print len([M for M in pathmatroids if M.size() == 8])
         print len([M for M in pathmatroids if M.size() == 9])
         print len([M for M in pathmatroids if M.size() == 10])
         print len([M for M in pathmatroids if M.size() == 11])
         print len([M for M in pathmatroids if M.size() == 12])
```

```
1999
0
4
20
93
380
1502
```

```
In [84]: for M in pathmatroids:
             if not any(M.has_minor(N) for N in [X8, Y8, Y8d, cat[8][6]]):
                 print "Counterexample: "
                 print M
         print "Done"
```

```
Done
```

```
In [85]: S = [cat[9][3], cat[9][13], cat[9][4], cat[9][14], cat[9][6], cat[9][1
         6]]
         T = [F7mm, F7mmd, P8m]
         TenEltList = []
         for M in S:
             TenEltList.extend(M.linear_extensions(element=len(M), simple=True)
         )
         TenEltList = [M for M in TenEltList if is_fragile(M) and not any(M.has
         _minor(N) for N in T)]
         TenEltList.extend([M.dual() for M in TenEltList])
         TenEltList = get_nonisomorphic_matroids(TenEltList)
         print len(TenEltList)
```

```
36
```

```
In [86]: for i in range(36):
             M = TenEltList[i]
             if not any(M.is_isomorphic(N) for N in pathmatroids):
                 print "Counterexample: "
                 print i
         print "Done"
```

```
Done
```

```
In [87]: ElevenEltList = []
         for M in TenEltList:
             ElevenEltList.extend(M.linear_extensions(element=len(M), simple=Tr
         ue))
         ElevenEltList = [M for M in ElevenEltList if is_fragile(M) and not any
         (M.has_minor(N) for N in T)]
         ElevenEltList.extend([M.dual() for M in ElevenEltList])
         ElevenEltList = get_nonisomorphic_matroids(ElevenEltList)
         print len(ElevenEltList)

         90
```

```
In [88]: for i in range(len(ElevenEltList)):
             M = ElevenEltList[i]
             if not any(M.is_isomorphic(N) for N in pathmatroids):
                 print "Counterexample: "
                 print i
         print "Done"

         Done
```

```
In [89]: TwelveEltList = []
         for M in ElevenEltList:
             TwelveEltList.extend(M.linear_extensions(element=len(M), simple=Tr
         ue))
         TwelveEltList = [M for M in TwelveEltList if is_fragile(M) and not any
         (M.has_minor(N) for N in T)]
         TwelveEltList.extend([M.dual() for M in TwelveEltList])
         TwelveEltList = get_nonisomorphic_matroids(TwelveEltList)
         print len(TwelveEltList)

         255
```

```
In [90]: for i in range(len(TwelveEltList)):
             M = TwelveEltList[i]
             if not any(M.is_isomorphic(N) for N in pathmatroids):
                 print "Counterexample: "
                 print i
         print "Done"

         Done
```


\end{document}